\documentclass{amsart}
\usepackage{amssymb,amsmath}
\usepackage{graphicx}
\usepackage{subfigure}
\usepackage{upgreek}
\usepackage{tikz}
\usepackage{tikz-cd}
\usepackage{hyperref}
\usepackage{mathtools}
\usetikzlibrary{matrix}
\usepackage{cmap}
\usepackage[T1]{fontenc}

\usepackage[utf8]{inputenc}
\usetikzlibrary{matrix}
\usepackage[all]{xy}
\usepackage{optparams,eurosym}
\usepackage{amsthm}
\usepackage{amsmath,amssymb}
\usepackage{dbnsymb}
\newtheorem{theorem}{Theorem}[section]
\newtheorem{example}[theorem]{Example}

\newtheorem{corollary}[theorem]{Corollary}
\newtheorem{lemma}[theorem]{Lemma}

\newtheorem{proposition}[theorem]{Proposition}
\newtheorem{definition}[theorem]{Definition}
\newtheorem{non-theorem}{Non-Theorem}

\theoremstyle{remark}
\newtheorem{remark}[theorem]{Remark}

\newcommand\abs[1]{|#1|}
\newcommand{\q}{\text{qdeg}}
\newcommand{\f}{\mathbb{F}}

\newcommand{\rp}[1]{\mathbb{RP}^{#1}}

\begin{document}
	\title {Bar-Natan homology for null homologous links in $\rp{3}$}
	\author{Daren Chen}
	\maketitle
	\begin{abstract}
		In this paper, we introduce Bar-Natan homology for null homologous links in $\mathbb{RP}^3$ over the field of two elements. It is a deformation of the Khovanov homology in $\mathbb{RP}^3$ defined by Asaeda, Przytycki and Sikora. We also define an $s$-invariant from this deformation using the same recipe as for links in $S^3$, and prove some genus bound using it. The key ingredient is the notion of 'twisted orientation' for null homologous links and cobordisms in $\mathbb{RP}^3$. 
		
	\end{abstract}
	\section{Introduction}
	
	Recently, Manolescu and Willis introduced the Lee homology and the associate an $s$-invariant for links in $\rp{3}$ in \cite{manolescu2023rasmussen}. This is a deformation of the Khovanov homology in $\rp{3}$ over the base field $\mathbb{Z}$, first introduced by Asaeda, Przytycki and Sikora in \cite{MR2113902} for fields of characteristic $2$, and Gabrov\v{s}ek in \cite{MR3189291} for the base field $\mathbb{Z}$ by fixing some sign conventions.
	
	As for the usual Lee homology, the construction in \cite{manolescu2023rasmussen} works as long as the characteristic of the base field is not $2$. For links in $S^3$, the solution for fields of characteristic $2$ is to use the Bar-Natan deformation of the Khovanov homology,  introduced by Bar-Natan in \cite{MR2174270},  instead of the Lee deformation. Here, we will do the same for null homologous links in $\rp{3}$. Similar to the $S^3$ case, the homology itself admits an easy description in terms of the number of components of the links. (Note that $H_1({\rp{3},\mathbb{Z}})= \mathbb{Z}_2$, the property that $\left[L\right] =0 \in H_1({\rp{3},\mathbb{Z}})$ doesn't depend on the choice of orientation on components of $L$, so we can talk about null homologous links without specifying the orientations of $L$.)
	
	\begin{theorem}
		For a twisted oriented null homologous link $L \subset \rp{3}$, one can associate a Bar-Natan chain complex $\mathit{CBN}(L)$ over the base field $\mathbb{F} = \mathbb{F}_2$, whose homology $\mathit{HBN}(L)$ is an invariant of twisted oriented $L$ as a bigraded vector space. More specifically, we have 
\begin{equation*}
	dim(\mathit{HBN}(L)) =
	\begin{cases}
		0 & \text{if $L$ has some component which is non-zero in $H_1(\rp{3},\mathbb{Z})$;}\\
		2^{|L|} & \text{otherwise, $i.e.,$ all components of $L$ are null homologous,}\\
	\end{cases}       
\end{equation*} and there is a basis of $\mathit{HBN}(L)$ given by $ \left\{s_o\mid o \text{ is a twisted orientation of } L\right\}$.
\label{thm:bar natan homology}
	\end{theorem}
See Definition \ref{def:twisted orientation} \ref{def:twisted orientation on links} and Lemma \ref{lemma:twisted orientation} for the notion of twisted orientation on null homologous links in $\rp{3}$. In short, a twisted orientation is an assignment of arrows on each segment of the link projection in $\rp{2}$, which is reversed each time it crosses a fixed essential unknot. More canonically, it is an orientation on the double cover $\widetilde{L}$ in $S^3$ of $L$ which is reversed by the deck transformation on $S^3$ of the covering map $S^3\to\rp{3}$. In particular, if $L$ has some homologically essential components, then there is no such twisted orientation on $L$, and $\mathit{HBN}(L)$ vanishes in this case.

The reason of the appearance of the twisted orientation is when we perform the Bar-Natan deformation, we are forced to assign the nontrivial map $id_V$ to the $1$-$1$ bifurcation in the definition of the Bar-Natan chain complex. This is different from the Lee deformation and the usual Khovanov homology in $\rp{3}$, where the map assigned to the $1$-$1$ bifurcation is the $0$ map. Therefore, we introduce the extra twisting to make the $1$-$1$ bifurcation behave more or less in the same way as $1$-$2$ and $2$-$1$ bifurcations. By using the notion of 'twisted orientation' instead of 'orientation', essentially all the proofs for the usual Bar-Natan/Lee homology in $S^3$ as in \cite{MR2173845}, \cite{MR2729272} and \cite{MR2286127} work in our setting with minor changes. 

As expected, since $dim(\mathit{HBN}(K)) =2$ for a null homologous knot, one can define an $s$-invariant, denoted as $s^{BN}_{\rp{3}}(K)$, from the quantum filtration on $CBN(K)$, and use it to establish a genus bound in the usual way. See Definition \ref{def:smin and smax} and \ref{def:s} for the definition of $s^{BN}_{\rp{3}}(K)$. It's worth noting that unlike the usual case where the genus bound is for the class of orientable slice surfaces, here we obtain a genus bound for twisted orientable slice surfaces. See Definition \ref{def:twisted orientable cobordism} and \ref{def:twisted orientable slice surface} for the precise definitions of twisted orientable cobordisms and slice surfaces. Roughly speaking, it means the double cover of the surface in $S^3\times I$ is orientable, and the fiberwise deck transformation $\tau $ in the $S^3$-direction reverses the orientation. We have the following relation between $s^{BN}_{\rp{3}}$ and the Euler characteristics of the twisted orientable slice surface, which is similar to the usual statement for the $s$-invariant in $S^3$. 

\begin{theorem}
Suppose $\Sigma:L\to L'$ is a twisted orientable cobordism between the null homologous links $L$ and $L'$ in $\rp{3} \times I$, then one can define a filtered chain map of filtration degree $\chi(\Sigma)$, \[F_{\Sigma}:\mathit{CBN}_{*,*}(L) \to \mathit{CBN}_{*,*}(L'),\] such that \[F_{\Sigma}(\left[s_o\right]) = \sum_{\left\{o_i\right\}}\left[s_{o_{i_{|L'}}}\right].\]
Here, $o$ is a twisted orientation on $L$, $\left\{o_i\right\}$ is the set of twisted orientations on $\Sigma$ that restrict to $o$ on $L$, $o_{i_{|L'}}$ represents the restriction of such orientations on $L'$, and $s_o$ (respectively $s_{o_{i_{|L'}}}$) are the corresponding canonical generators of $\mathit{HBN}(L)$ (respectively $\mathit{HBN}(L')$) defined in Definition \ref{def:canonical generator}.

In particular, when both $L$ and $L'$ are null homologous knots, $F_{\Sigma}$ is a quasi-isomorphism of filtration degree $\chi(\Sigma)$. Further specifying to the case where $L'$ is the trivial unknot in $\rp{3}$, one gets 
\[-\chi(\Sigma) \geq \abs{s^{BN}_{\rp{3}}(L)},\] for any twisted orientable slice surface $\Sigma$ of the knot $L$.
\label{thm:genus bound}
\end{theorem}
A twisted orientable slice surface could be either orientable or unorientable, and not every (un)orientable slice surface is necessarily twisted orientable. See Example \ref{ex:twisted orientable} for different possibilities of the combination of (not) twisted orientable/(un)orientable cobordisms. 

The $s$-invariant $s^{BN}_{\rp{3}}(K)$ shares similar formal properties with the usual $s$-invariant in $S^3$. See Proposition \ref{prop:mirroring} for mirroring, Proposition \ref{prop:local connected sum} for taking a local connected sum with a local knot and Proposition \ref{prop:positive} for 'positive knots' in the sense of twisted orientations. All of these follow from the same arguments as in the case of $S^3$. We also give a simple example to demonstrate the difference between the $s$-invariant $s^{BN}_{\rp{3}}$ defined in this paper and the $s$-invariant defined as in \cite{manolescu2023rasmussen}. Consider the following know as in Figure \ref{fig:example_difference}, which is 'positive' with respect to the twisted orientation, so we have \[s^{BN}_{\rp{3}}(K) =3. \] But one can construct an orientable slice surface of this knot with genus $1$, so the $s$-invariant of it defined in \cite{manolescu2023rasmussen} satisfies \[\abs{s(K)} \leq 2.\] In addition, we provide a family of knots in which their two $s$-invariants differ by an arbitrarily large amount, obtained by inserting more and more full twists in this example.See Example \ref{ex:a calculation} for a more detailed discussion.

	\begin{figure}[h]
	\[{
		\fontsize{7pt}{9pt}\selectfont
		\def\svgscale{0.6}
		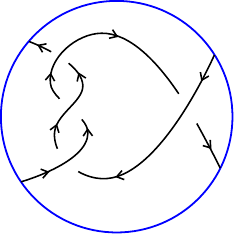
	}\]
	\caption{An example of difference in $s$-invariants defined using Lee deformation and Bar-Natan deformation}
	\label{fig:example_difference}
\end{figure}
\textbf{Organization of the paper} In Section \ref{sec:Bar Natan homology}, we define the Bar-Natan deformation of the Khovanov complex for null homologous links in $\rp{3}$ over the field $\mathbb{F}_2$. We also introduce the notion of twisted orientation and prove Theorem \ref{thm:bar natan homology}. In Section \ref{sec:cobordism}, we define cobordism maps on the Bar-Natan chain complex for twisted orientable cobordisms and establish the first statement of Theorem \ref{thm:genus bound}. In Section \ref{sec: invariant}, we define the Bar-Natan $s$-invariant $s^{BN}_{\rp{3}}$, discuss its formal properties and complete the proof of Theorem \ref{thm:genus bound}. In Section \ref{sec:further directions}, we discuss some further directions based on this work.

\textbf{Acknowledgements.} The author would like to thank Cole Hugelmeyer, Robert Lipshitz, and Paul Wedrich and Mike Willis for helpful conversations. The author wants to express his appreciation particularly to his advisor Ciprian Manolescu for all the guidance, support and encouragement during this project and all through the author's graduate school. This work constitutes a part of the author's PhD thesis.

This work is partially  supported by NSF grant DMS-2003488.

	\section{Bar-Natan homology for null homologous links in $\rp{3}$}
	\label{sec:Bar Natan homology}
	Throughout the paper, we work over the field $\mathbb{F} = \mathbb{F}_2$ of two elements, unless otherwise specified.
	
	Let $V=\mathbb{F}\langle1,x\rangle$ be the graded vector space generated by $1,x$, with quantum grading given by \[\q(1)=1, \,\,\q(x) = -1.\]
	\begin{definition}
		 The \textbf{Bar-Natan Frobenius algebra structure on $V$} is a deformation of the usual Frobenius structure on $V=H^*(S^2)$. It consists of the tuple $\left(V,m,\Delta,\iota,\eta\right)$, where the multiplication $m:V\otimes V\to V$ is given by \[1\otimes 1\to 1,\,\,\,\,\,\,\,\, 1\otimes x\to x,\,\,\,\,\,\,\,\, x\otimes 1 \to x,\,\,\,\,\,\,\,\,x\otimes x\to x,\] the comultiplication $\Delta:V \to V\otimes V$ is given by \[1\to 1\otimes x + x\otimes 1 + 1\otimes 1,\,\,\,\,\,\,\,\, x\to x\otimes x, \]
		 the unit $\iota:\mathbb{F} \to V$ is given by \[1\to 1,\] and the counit $\eta:V\to \mathbb{F}$ is given by \[1\to 0,\,\,\,\,\,\,\,\,x\to 1.\]
		 \label{def:frobenius algebra}
	\end{definition}
\begin{remark}
	Some signs in the comultiplication $\Delta$ might look different from the usual convention, but it makes no difference as we are working over the field $\f$ of characteristic $2$. We also use the version in which the extra deformation variable of the Bar-Natan deformation is set to $1$, so all the structure maps are filtered instead of being grading-preserving after suitable shifting in the quantum grading.
\end{remark}

\begin{definition}
	A link $L$ in $\rp{3}$ is \textbf{null homologous} if \[\left[L\right] = 0\in H_1(\rp{3},\mathbb{Z}) = \mathbb{Z}/2.\]
	If particular, we don't require $L$ to be oriented, as different orientations won't change $[L]$ in $H_1(\rp{3},\mathbb{Z})$, which is $\mathbb{Z}/2$.
\end{definition}
    Let $L$ be a null homologous link in $\rp{3}$. Suppose $L$ is away from some fixed point $*\in \rp{3}.$ Then, we obtain a link projection diagram $D$ of $L$ in $\rp{2}$, using the twisted $I$-bundle structure $\rp{3}\backslash \left\{*\right\} \cong \rp2 \widetilde{\times }I$. Suppose $D$ has $n$ crossings. By choosing an ordering on the crossings, we form the usual cube  \[\underline{2}^n := (0\rightarrow1)^n\] of smoothings $D_v$ of $D$ for each vertex $v\in \underline{2}^n,$ following the convention of $0$- and $1$-smoothings as indicated in Figure \ref{fig:smoothing}.
    
    	\begin{figure}[h]
    	\[
    	{
    		\fontsize{7pt}{9pt}\selectfont
    		\def\svgscale{0.3}
\begingroup%
  \makeatletter%
  \providecommand\color[2][]{%
    \errmessage{(Inkscape) Color is used for the text in Inkscape, but the package 'color.sty' is not loaded}%
    \renewcommand\color[2][]{}%
  }%
  \providecommand\transparent[1]{%
    \errmessage{(Inkscape) Transparency is used (non-zero) for the text in Inkscape, but the package 'transparent.sty' is not loaded}%
    \renewcommand\transparent[1]{}%
  }%
  \providecommand\rotatebox[2]{#2}%
  \newcommand*\fsize{\dimexpr\f@size pt\relax}%
  \newcommand*\lineheight[1]{\fontsize{\fsize}{#1\fsize}\selectfont}%
  \ifx\svgwidth\undefined%
    \setlength{\unitlength}{589.01869859bp}%
    \ifx\svgscale\undefined%
      \relax%
    \else%
      \setlength{\unitlength}{\unitlength * \real{\svgscale}}%
    \fi%
  \else%
    \setlength{\unitlength}{\svgwidth}%
  \fi%
  \global\let\svgwidth\undefined%
  \global\let\svgscale\undefined%
  \makeatother%
  \begin{picture}(1,0.266363)%
    \lineheight{1}%
    \setlength\tabcolsep{0pt}%
    \put(0,0){\includegraphics[width=\unitlength,page=1]{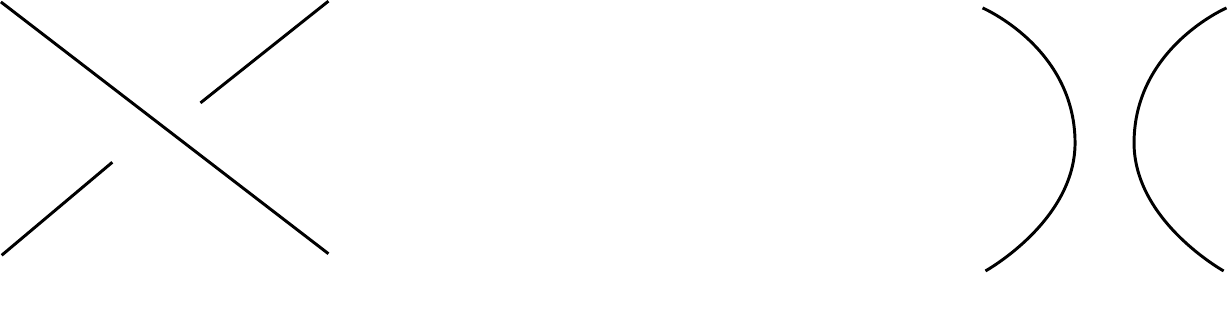}}%
    \put(0.40,-0.01519316){\color[rgb]{0,0,0}\makebox(0,0)[lt]{\lineheight{1.25}\smash{\begin{tabular}[t]{l}0-Smoothing\end{tabular}}}}%
    \put(0.8,-0.01526631){\color[rgb]{0,0,0}\makebox(0,0)[lt]{\lineheight{1.25}\smash{\begin{tabular}[t]{l}1-Smoothing\end{tabular}}}}%
    \put(0,0){\includegraphics[width=\unitlength,page=2]{smoothing.pdf}}%
  \end{picture}%
\endgroup%

    	}
    	\]
    	\caption{$0$ and $1$- Smoothings}
    	\label{fig:smoothing}
    \end{figure}
    Since we start with a null homologous link and the smoothing process doesn't change the homology class, it is easy to see that each component in the smoothing $D_v$ is a local unknot in $\rp{2}$, for every vertex $v\in\underline{2}^n$. We will apply the Bar-Natan Frobenius algebra to this cube of smoothings, $i.e.$, we associate the vector space $V$ to each local unknot and then taking tensor product for each smoothing $D_v$. 
    
    As for the edge maps, they correspond to change $0$-smoothing to $1$-smoothing at one crossing. Unlike the usual case for link projection in $\mathbb{R}^2$, there are three possibilities as indicated in Figure \ref{fig:bifurcation}
    : $1$-$2$ bifurcation, which splits one circle into two circles; $2$-$1$ bifurcation, which merges two circles into one circle; and $1$-$1$ bifurcation, which twists one circle into another circle. Throughout th the paper, we will draw $\rp{2}$ as a disk with half of its boundary (the blue circle) identified with the other half. We assign the multiplication map \[m:V\otimes V\to V\] to the $2$-$1$ bifurcation and the comultiplication map \[\Delta: V \to V\otimes V\] to the $1$-$2$ bifurcation map as usual. For the $1$-$1$ bifurcation, it should correspond to a map \[f:V\to V,\] which is not included in the Bar-Natan Frobenius algebra structure of $V$ as defined above. It turns out that there is a unique choice of $f$ over the base field $\mathbb{F}_2$ which gives a filtered chain complex if we feed the cube of smoothings by the Frobenius algebra $V$, and use $m,\Delta, f$ for the edge maps.
    	\begin{figure}[t]
    	\[
    	{
    		\fontsize{7pt}{9pt}\selectfont
    		\def\svgwidth{5.5in}
\begingroup%
  \makeatletter%
  \providecommand\color[2][]{%
    \errmessage{(Inkscape) Color is used for the text in Inkscape, but the package 'color.sty' is not loaded}%
    \renewcommand\color[2][]{}%
  }%
  \providecommand\transparent[1]{%
    \errmessage{(Inkscape) Transparency is used (non-zero) for the text in Inkscape, but the package 'transparent.sty' is not loaded}%
    \renewcommand\transparent[1]{}%
  }%
  \providecommand\rotatebox[2]{#2}%
  \newcommand*\fsize{\dimexpr\f@size pt\relax}%
  \newcommand*\lineheight[1]{\fontsize{\fsize}{#1\fsize}\selectfont}%
  \ifx\svgwidth\undefined%
    \setlength{\unitlength}{1002.50571894bp}%
    \ifx\svgscale\undefined%
      \relax%
    \else%
      \setlength{\unitlength}{\unitlength * \real{\svgscale}}%
    \fi%
  \else%
    \setlength{\unitlength}{\svgwidth}%
  \fi%
  \global\let\svgwidth\undefined%
  \global\let\svgscale\undefined%
  \makeatother%
  \begin{picture}(1,0.13054486)%
    \lineheight{1}%
    \setlength\tabcolsep{0pt}%
    \put(0.482,0.0705259){\color[rgb]{0,0,0}\makebox(0,0)[lt]{\lineheight{1.25}\smash{\begin{tabular}[t]{l}$1$-$2$\end{tabular}}}}%
    \put(0,0){\includegraphics[width=\unitlength,page=1]{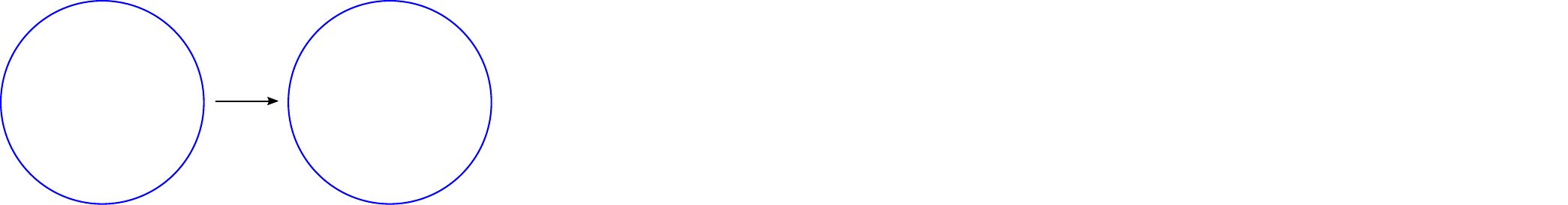}}%
    \put(0.142,0.06909197){\color[rgb]{0,0,0}\makebox(0,0)[lt]{\lineheight{1.25}\smash{\begin{tabular}[t]{l}$2$-$1$      \end{tabular}}}}%
    \put(0,0){\includegraphics[width=\unitlength,page=2]{bifurcation.pdf}}%
    \put(0.832,0.06977455){\color[rgb]{0,0,0}\makebox(0,0)[lt]{\lineheight{1.25}\smash{\begin{tabular}[t]{l}$1$-$1$\end{tabular}}}}%
  \end{picture}%
\endgroup%

    	}
    	\]
    	\caption{$2$-$1$, $1$-$2$ and $1$-$1$ bifurcations}
    	\label{fig:bifurcation}
    \end{figure}
    \begin{lemma}
    	\label{lem:1-1bifurcation}
    	The only possible choice of $f:V\to V$ for the $1$-$1$ bifurcation map over $\mathbb{F}_2$ that gives a chain complex such that $f:V\to V\left[1\right]$ is filtered is \[f= id_V.\] Here $V\left[1\right]$ is the graded vector space obtained from $V$ by shifting up by $1$ in the quantum grading.
    \end{lemma}
		\begin{figure}[h]
	\[
	{
		\fontsize{7pt}{9pt}\selectfont
		\def\svgwidth{4in}
		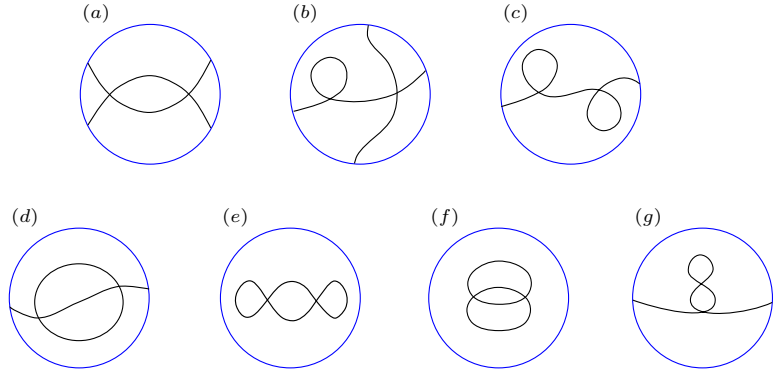
	}
	\]
	\caption{Singular graphs in $\rp{2}$ with 2 singular points. See \cite{MR3189291}}
	\label{fig:singular-curves-with-2-singularity}
\end{figure}
\begin{proof}
	To get a chain complex, we want each face of the cube of smoothings to give a commutative diagram. It amounts to checking that each link diagram in $\rp{2}$ with $2$ crossings gives a commutative square. A list of link diagrams in $\rp{2}$ is obtained from the singular diagrams in Figure \ref{fig:singular-curves-with-2-singularity} by replacing each singular point by an over or under crossing. As we are considering null homologous links, only $(a), (b), (e),(f)$ are possible. A typical square of smoothings and the corresponding maps are shown in Figure \ref{fig:requirement_of_f}.

			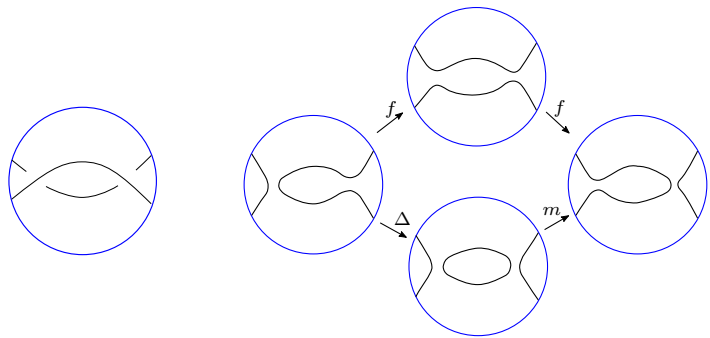
\begin{figure}[h]
		\[
		{
			\fontsize{7pt}{9pt}\selectfont
			\def\svgscale{0.5}
\begingroup%
  \makeatletter%
  \providecommand\color[2][]{%
    \errmessage{(Inkscape) Color is used for the text in Inkscape, but the package 'color.sty' is not loaded}%
    \renewcommand\color[2][]{}%
  }%
  \providecommand\transparent[1]{%
    \errmessage{(Inkscape) Transparency is used (non-zero) for the text in Inkscape, but the package 'transparent.sty' is not loaded}%
    \renewcommand\transparent[1]{}%
  }%
  \providecommand\rotatebox[2]{#2}%
  \newcommand*\fsize{\dimexpr\f@size pt\relax}%
  \newcommand*\lineheight[1]{\fontsize{\fsize}{#1\fsize}\selectfont}%
  \ifx\svgwidth\undefined%
    \setlength{\unitlength}{526.84362003bp}%
    \ifx\svgscale\undefined%
      \relax%
    \else%
      \setlength{\unitlength}{\unitlength * \real{\svgscale}}%
    \fi%
  \else%
    \setlength{\unitlength}{\svgwidth}%
  \fi%
  \global\let\svgwidth\undefined%
  \global\let\svgscale\undefined%
  \makeatother%
  \begin{picture}(1,0.47130543)%
    \lineheight{1}%
    \setlength\tabcolsep{0pt}%
    \put(0,0){\includegraphics[width=\unitlength,page=1]{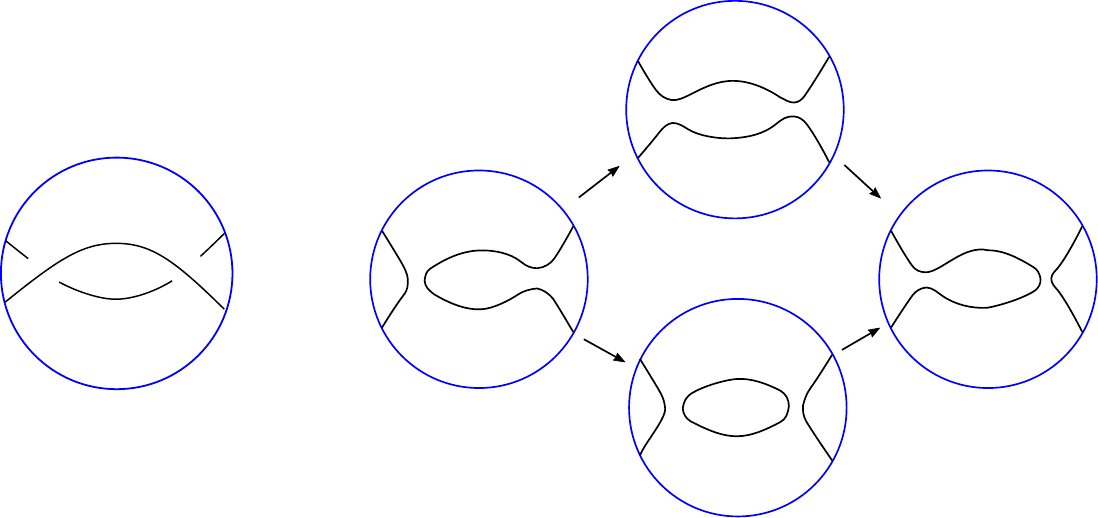}}%
    \put(0.53847355,0.318){\color[rgb]{0,0,0}\makebox(0,0)[lt]{\lineheight{1.25}\smash{\begin{tabular}[t]{l}$f$\end{tabular}}}}%
    \put(0.78,0.32){\color[rgb]{0,0,0}\makebox(0,0)[lt]{\lineheight{1.25}\smash{\begin{tabular}[t]{l}$f$\end{tabular}}}}%
    \put(0.5514897,0.15919511){\color[rgb]{0,0,0}\makebox(0,0)[lt]{\lineheight{1.25}\smash{\begin{tabular}[t]{l}$\Delta$\end{tabular}}}}%
    \put(0.76410774,0.173){\color[rgb]{0,0,0}\makebox(0,0)[lt]{\lineheight{1.25}\smash{\begin{tabular}[t]{l}$m$\end{tabular}}}}%
  \end{picture}%
\endgroup%

		}
		\]
		\caption{One example of a square in the cube of resolution}
		\label{fig:requirement_of_f}
	\end{figure}
	 The condition that these squares being commutative gives the following restrictions on the map $f:V\to V$ :
	\begin{enumerate}
		\item $f^2 = m\circ \Delta = id_V,$
		\item $m\circ (f\otimes id_V) = m\circ (id_V\otimes f) = f\circ m,$
		\item $(f\otimes id_V) \circ\Delta = (id_V \otimes f)\circ \Delta = \Delta \circ f.$
	\end{enumerate}

Also, in defining the chain complex, we will apply a shift in the quantum grading of degree \newline  $|v| = \#$ of $1$'s in $v\in \underline{2}^n$ to the vector space corresponding to the smoothing $D_v$. Therefore, to make the edge map filtered, \textit{i.e.}, non-decreasing in the quantum grading, we also require:\[f:V\to V\left[1\right] \text{ to be filtered.}\]

Now it is easy to see the only possible choice of $f$ satisfying the above conditions is \[f=id_V.\]
\end{proof}

\begin{remark}
	
	In \cite{MR2113902}, Asaeda, Przytycki and Sikora extended the Khovanov homology to links in $\rp{3}$ over the field $\mathbb{F}_2$, where they assigned the $0$ map to the $1$-$1$ bifurcation map. Because they use the usual Frobenius algebra for Khovanov homology, the corresponding condition for $f$ becomes $f^2=m\circ \Delta =0$, so $f=0$ is a natural choice. In \cite{MR3189291}, Gabrov\v{s}ek introduced some sign conventions to make the extension of Khovanov homology work over field of characteristic $0$, where the $1$-$1$ bifurcation map is also assigned $0$. More recently, in \cite{manolescu2023rasmussen} Manolescu and Willis extended the definition of Lee homology to links in $\rp{3}$, where, again, the $1$-$1$ bifurcation map is assigned $0$. The difference in the assignment in the $1$-$1$ bifurcation map finally leads to the difference in the class of slice surfaces that we can control the genus from the $s$-invariant construction in this paper and in \cite{manolescu2023rasmussen}.

In \cite{chen2021khovanovtype}, the author defined some variation of the usual Khovanov homology in $\rp{3}$. One can try apply the Bar-Natan Frobenius algebra structure on $V$ instead of the usual Khovanov Frobenius algebra structure and see what happens. Unfortunately, it won't give more interesting homology theory. The reason is that if one uses the Bar-Natan Frobenius algebra structure, then the requirement for $f$ and $g$ becomes $f\circ g = id_{V_0}, g\circ f = id_{V_1}$, which implies $V_0$ and $V_1$ are isomorphic, and the chain complex will just be a direct sum of $n$ copies of the reduced version of the Bar-Natan chain complex defined in this paper, where $n = \text{dim}(V_0) = \text{dim}(V_1)$. In the Khovanov setting, the equations we need for $f$ and $g$ are $f\circ g = g \circ f=0$, which leaves more room for the choice of $ V_0, V_1, f$ and $g$. Still, one could get different spectral sequences for different choices of $V_0,V_1,f,$ and $g$. We haven't attempted to explore the possibilities in this project.

\end{remark}
By examining the commutativity of the remaining possible $2$-crossing diagrams in $\rp{2}$ which doesn't involve the $1$-$1$ bifurcation (these are local diagrams, which means the check is the same as checking if the Bar-Natan Frobenius algebra structure on $V$ gives a chain complex for links in $S^3$ ), we can define the following chain complex for the link diagram $D$ in $\rp{2}$.

\begin{definition}
	Given an $n$-crossing link diagram $D$ in $\rp{2}$ of a null homologous link, the \textbf{unadjusted Bar-Natan chain complex $\mathit{CBN}_{*,*}^{un}(D)$}  is the bigraded chain complex  
	\[\mathit{CBN}_{i,*}^{un} = \bigoplus_{v\in \underline{2}^n,|v|=i}C(D_v)[i], \,\,\,\,\,\,\,\,\partial = \sum _{e:\text{ edges in  }\underline{2}^n}\partial_e,\]
	where the first grading is the homological grading, and the second grading is the quantum grading. Denote its homology by $\mathit{HBN}^{un} (D)$. Here $C(D_v)= V^{\otimes k_v}$ if the smoothing $D_v$ consists of $k_v$ unknots, $\left[i\right]$ denotes shifting up by $i$ in the quantum grading, and $\partial_e$ applies $m,\Delta $ or $f$ to the involved unknots, depending on whether the edge is a $2$-$1$, $1$-$2$ or $1$-$1$ bifurcation.
	\label{def:unadjusted chain complex}
\end{definition}

It is called the 'unadjusted' Bar-Natan complex as we will do some global shifts in the homological and quantum grading to make it a link invariant, as in the usual definition for Khovanov/Bar-Natan homology in $S^3$. However, the way we do the shift is a bit unconventional, so we delay the discussion until later when we introduce more terminology.

Note that the quantum grading-preserving part of the differential is exactly the same as the differential used in the definition of Khovanov homology for links in $\rp{3}$ over $\mathbb{F}_2$ as in \cite{MR2113902}, where the map associated with the $1$-$1$ bifurcation is the $0$ map. (Note that such a map needs to be grading-preserving from $V$ to $V\left[1\right]$, so $id_V$ doesn't preserve the quantum grading)

Now the natural next step is to check whether the homology only depends on the link rather than the link diagram. This is usually done by checking the invariance of the homology under Reidemeister moves. As discussed in \cite{MR1296890}, for link projections in $\rp{2}$, there is a similar list of Reidemeister moves that relate different projections of isotopic links in $\rp{3}$, as shown in Figure \ref{fig:r-moves}. However, for arguments in the later sections involving $s$-invariant, we need stronger conditions on the induced map by Reidemeister moves, so we again delay the discussion of invariance until later. 

	\begin{figure}[h]
	\[
	{
		\fontsize{7pt}{9pt}\selectfont
		\def\svgscale{0.4}
		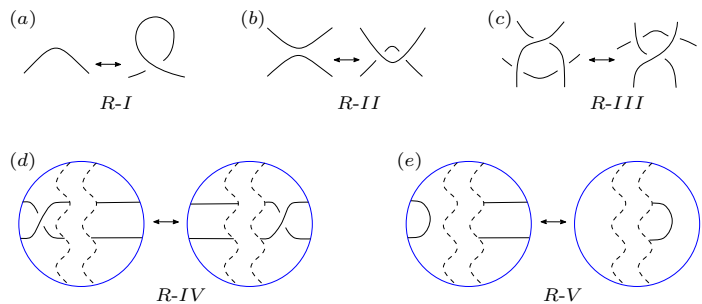
	}
	\]
	\caption{Reidemeister moves in $\rp{3}$}
	\label{fig:r-moves}
\end{figure}

Recall that in the usual Bar-Natan homology for links in $S^3$, the dimension of Bar-Natan homology is actually determined just by the number of components in the link, and there is a canonical basis of the homology corresponding to different choices of orientations on each component of the link. Furthermore, the induced maps by the Reidemeister moves will send this canonical basis to the corresponding ones after the Reidemeister moves. We will prove a similar result for the Bar-Natan homology for null homologous links in $\rp{3}$: the dimension of the homology is determined by the number of null homologous components of the link, a canonical basis of the homology is given by 'twisted orientations' on the link, and Reidemeister moves send the canonical basis to the canonical basis.

For that, we need to introduce the notion of 'twisted orientation'. This is the central notion of the paper. Essentially, all the proof for the usual Bar-Natan homology in $S^3$ works in $\rp{3}$ if we replace 'orientation' in $S^3$ by 'twisted orientation' in $\rp{3}$. 
\begin{definition}
	 Denote $U_1$ as the essential unknot in $\rp{2}$, which is the quotient of the boundary of the disk. In this paper, we draw $\rp{2}$ as identifying the two halves of the boundary of a disk, and $U_1$ is represented by the quotient of the blue circles in all the figures.
\end{definition}
\begin{definition}
Suppose $D$ is a link diagram in $\rp{2}$ such that each component of the link $L$ represented by $D$ in $\rp{3}$ is null homologous. A \textbf{twisted orientation} on $D$ is an assignment of arrows on each segment of $D$, such that it is reversed each time it crosses $U_1$. 
	\label{def:twisted orientation}
\end{definition}
See Figure \ref{fig:twisted-orientation} for an illustration of the twisted orientation, and a comparison with the usual orientation. Note that a twisted orientation only exists if we assume each component of $D$ is null homologous since we can't give an alternating assignment of arrows on a homologically essential component, which intersects $U_1$ an odd number of times. 

More canonically, we have another way to view the twisted orientation, which does not depend on the link diagram $D$.
\begin{definition}
	Suppose $L$ is a link in $\rp{3}$ such that each component of $L$ is null homologous. Let $\widetilde{L}$ denote the double cover of $L$ in $S^3$, and $\tau$ denote the deck transformation of the covering map $S^3\to \rp{3}$. A \textbf{twisted orientation} on $L$ is an orientation on $\widetilde{L}$ which is reversed under the action of $\tau$.
	\label{def:twisted orientation on links}
\end{definition}
\begin{figure}[h]
	\[{
		\fontsize{9pt}{11pt}\selectfont
		\def\svgscale{0.5}
\begingroup%
  \makeatletter%
  \providecommand\color[2][]{%
    \errmessage{(Inkscape) Color is used for the text in Inkscape, but the package 'color.sty' is not loaded}%
    \renewcommand\color[2][]{}%
  }%
  \providecommand\transparent[1]{%
    \errmessage{(Inkscape) Transparency is used (non-zero) for the text in Inkscape, but the package 'transparent.sty' is not loaded}%
    \renewcommand\transparent[1]{}%
  }%
  \providecommand\rotatebox[2]{#2}%
  \newcommand*\fsize{\dimexpr\f@size pt\relax}%
  \newcommand*\lineheight[1]{\fontsize{\fsize}{#1\fsize}\selectfont}%
  \ifx\svgwidth\undefined%
    \setlength{\unitlength}{335.61789868bp}%
    \ifx\svgscale\undefined%
      \relax%
    \else%
      \setlength{\unitlength}{\unitlength * \real{\svgscale}}%
    \fi%
  \else%
    \setlength{\unitlength}{\svgwidth}%
  \fi%
  \global\let\svgwidth\undefined%
  \global\let\svgscale\undefined%
  \makeatother%
  \begin{picture}(1,0.42949182)%
    \lineheight{1}%
    \setlength\tabcolsep{0pt}%
    \put(0,0){\includegraphics[width=\unitlength,page=1]{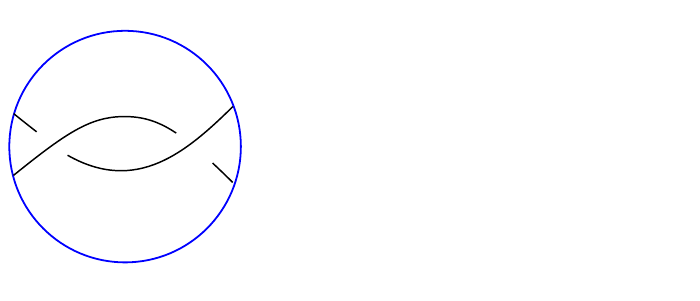}}%
    \put(-0.00114352,0.41169476){\color[rgb]{0,0,0}\makebox(0,0)[lt]{\lineheight{1.25}\smash{\begin{tabular}[t]{l}$(a)$\end{tabular}}}}%
    \put(0.66052039,0.41169534){\color[rgb]{0,0,0}\makebox(0,0)[lt]{\lineheight{1.25}\smash{\begin{tabular}[t]{l}$(b)$\end{tabular}}}}%
    \put(-0.03797239,-0.03){\color[rgb]{0,0,0}\makebox(0,0)[lt]{\lineheight{1.25}\smash{\begin{tabular}[t]{l}a twisted orientation \end{tabular}}}}%
    \put(0.64573158,-0.03){\color[rgb]{0,0,0}\makebox(0,0)[lt]{\lineheight{1.25}\smash{\begin{tabular}[t]{l}a usual orientation \end{tabular}}}}%
    \put(0,0){\includegraphics[width=\unitlength,page=2]{twisted_orientaion.pdf}}%
  \end{picture}%
\endgroup%

	}\]
	\caption{An example of a twisted orientation compared with a usual orientation on the same knot}
	\label{fig:twisted-orientation}
\end{figure}

The next lemma proves these two definitions agree, and when we refer to a twisted orientation on a link $L$ in $\rp{3}$, we mean either of the two, depending on the context.
\begin{lemma}
Suppose $D$ is a link diagram of a link $L$ in $\rp{3}$, such that each component of $L$ is null homologous. Then a twisted orientation on $D$ gives a twisted orientation on $L$, and vice versa. In particular, there are $2^{|L|}$ of twisted orientations on $L$, where $|L|$ is the number of components in $L$.
	\label{lemma:twisted orientation}
\end{lemma}
\begin{figure}[h]
	\[{
		\fontsize{7pt}{9pt}\selectfont
		\def\svgscale{0.5}
\begingroup%
  \makeatletter%
  \providecommand\color[2][]{%
    \errmessage{(Inkscape) Color is used for the text in Inkscape, but the package 'color.sty' is not loaded}%
    \renewcommand\color[2][]{}%
  }%
  \providecommand\transparent[1]{%
    \errmessage{(Inkscape) Transparency is used (non-zero) for the text in Inkscape, but the package 'transparent.sty' is not loaded}%
    \renewcommand\transparent[1]{}%
  }%
  \providecommand\rotatebox[2]{#2}%
  \newcommand*\fsize{\dimexpr\f@size pt\relax}%
  \newcommand*\lineheight[1]{\fontsize{\fsize}{#1\fsize}\selectfont}%
  \ifx\svgwidth\undefined%
    \setlength{\unitlength}{547.40152212bp}%
    \ifx\svgscale\undefined%
      \relax%
    \else%
      \setlength{\unitlength}{\unitlength * \real{\svgscale}}%
    \fi%
  \else%
    \setlength{\unitlength}{\svgwidth}%
  \fi%
  \global\let\svgwidth\undefined%
  \global\let\svgscale\undefined%
  \makeatother%
  \begin{picture}(1,0.44506295)%
    \lineheight{1}%
    \setlength\tabcolsep{0pt}%
    \put(0,0){\includegraphics[width=\unitlength,page=1]{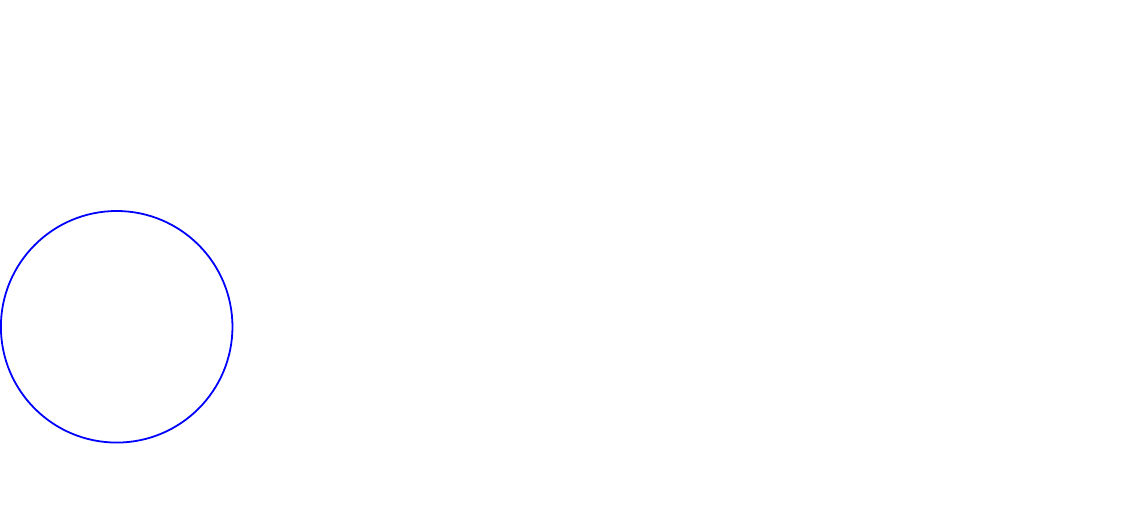}}%
    \put(0.09623242,0.00042603){\color[rgb]{0,0,0}\makebox(0,0)[lt]{\lineheight{1.25}\smash{\begin{tabular}[t]{l}$L$\end{tabular}}}}%
    \put(0,0){\includegraphics[width=\unitlength,page=2]{twisted_orientation_in_double_cover.pdf}}%
    \put(0.58,0.00531516){\color[rgb]{0,0,0}\makebox(0,0)[lt]{\lineheight{1.25}\smash{\begin{tabular}[t]{l}$\widetilde{L} = cl(T\circ F(T))$\end{tabular}}}}%
    \put(0,0){\includegraphics[width=\unitlength,page=3]{twisted_orientation_in_double_cover.pdf}}%
  \end{picture}%
\endgroup%

	}\]
	\caption{A twisted orientation on $L$ and the corresponding orientation on $\widetilde{L}$ which is reversed by the action of $\tau$}
	\label{fig:twisted-orientation_in_double_cover}
\end{figure}
\begin{proof}
	It is easy to see that a knot $K$ in $\rp{3}$ is null homologous if and only if its double cover $\widetilde{K}$ in $S^3$ is a $2$-component link (instead of a knot). Therefore, due to the null homologous assumption, each component of $L$ lifts to a two-component link in $\widetilde{L}$. 
	
	One way to draw a diagram of $\widetilde{L}$ is as follows. View the link diagram $D$ in $\rp{2}$ as an $n$-$n$ tangle, denoted by $T$. Let $F(T)$ be the $n$-$n$ tangle obtained by flyping $T$, i.e. rotating $T$ by $180^{\circ}$ about its middle horizontal axis in the plane. Then $\widetilde{L}$ is the closure of the composition $L\circ F(T)$, where the deck transformation $\tau$ acts by swapping $L$ and $F(T)$. See Figure \ref{fig:twisted-orientation_in_double_cover} for an example. 
	
	Now consider one specific component $K$ in $L$. Every time $K$ hits the essential unknot $U_1$, it travels from one copy of $\rp{2}$ to the other copy of $\rp{2}$ in the lifted picture of $\widetilde{K}$. Therefore, when we restrict to one copy of $\rp{2}$, two adjacent segments (adjacent in the sense that they share common points on $U_1$) must come from different components of the lifted link $\widetilde{L}$ (labeled by red and green in Figure \ref{fig:twisted-orientation_in_double_cover}, respectively). The requirement that the arrow is reversed in the definition of $D$ then becomes the requirement that the deck transformation $\tau$ reverses the orientations on the two components of $\widetilde{K}$.
	 
\end{proof}

\begin{remark}
	If we are given an orientation on the lifted link $\widetilde{L}$ that is reversed by $\tau$, and we present $\widetilde{L}$ as the closure of $T\circ F(T)$, there is an ambiguity in assigning the corresponding twisted orientation on the quotient link $L$, which arises from whether we use $T$ or $F(T)$ as a link diagram for $L$ in $\rp{2}$. After we fix a choice of the fundamental domain of the deck transformation $\tau$ ($i.e.$ choosing $T$ instead of $F(T)$), then there is a canonical one-to-one correspondence between the set of twisted orientations of $L$ and the set of orientations of $\widetilde{L}$ that are reversed by $\tau$.
	
\end{remark}

Now we are going to define elements in the Bar-Natan homology corresponding to the twisted orientations. Later, we will show they give a basis of the Bar-Natan homology, as in the case of Bar-Natan homology in $S^3$.

Recall that we can diagonalize the Bar-Natan Frobenius algebra $V$ using the basis $a = 1+x, b=x$ (remember we always work on field $\mathbb{F}_2$), such that the multiplication and comultiplication becomes:
\begin{align*}
	&\,\,m: V\otimes V \to V \\
	a\otimes a \to a, \,\,\,\, &b \otimes b \to b,\,\,\,\,  a\otimes b\to 0,\,\,\,\, b\otimes a \to 0	\\
	&\,\,\Delta: V\to V\otimes V\\
	&a\to a\otimes a,\,\,\,\, b\to b\otimes b.
\end{align*}
In the usual Bar-Natan homology in $S^3$, for each orientation of the link, one associates an element in the homology by forming the oriented resolutions and then assign either $a$ or $b$ to each unknot component in the resolution, depending on their orientations and distance from infinity. Here we will do similar things. Since a twisted orientation is just some assignment of arrows on each segment in the link diagram, and forming an oriented resolution only cares about the orientations near each crossing, we can form the 'twisted orientation resolution' of a twisted oriented link diagram by doing the oriented resolution at each crossing in the usual sense. Then we assign $a$ or $b$ to each unknot component in the twisted oriented resolution according to some rule. The rule needs some explanation, so let's do it first.

\begin{definition}

Suppose $D$ is a twisted orientation link diagram in $\rp{2}$ of an unlink $U$ in $\rp{3}$ with no crossings, $i.e.$ it is a disjoint union of local unknots. Pick a point $p$ on the essential unknot $U_1$, which is disjoint from the link diagram $D$. The choice of the point $p$ gives a way to view the diagram $D$ as an $n$-$n$ tangle $T$. We form the lifted link $\widetilde{U}$ as the closure of $T\circ F(T)$ as before. As discussed in Lemma \ref{lemma:twisted orientation}, a twisted orientation on $D$ is the same as an orientation on $\widetilde{U}$ which is reversed under $\tau$, in particular each component of $\widetilde{U}$ is oriented.

Define the following $\mathbb{Z}/2$-valued functions on components of $\widetilde{U}$:

Let \[d: \pi_0(\widetilde{U}) \to \mathbb{Z}/2\] be the number of circles in $\widetilde{L}$ separating the chosen component from infinity mod $2$. 

Let \[o:\pi_0(\widetilde{U}) \to \mathbb{Z}/2\] equal $1$ if the component is oriented counterclockwise, and $0$ if it is oriented clockwise. 

Let \[l:\pi_0(\widetilde{U}) \to \mathbb{Z}/2\] be the sum of the above two functions mod $2$: \[l = d+o \text{ mod }2.\]
\label{def:assignment of a and b}
\end{definition}
Now we want to prove that $l$ takes the same value on the two components in $\widetilde{U}$ which is mapped to the same component in $U$, so it descends to a map $l:\pi_0(U)\to \mathbb{Z}/2.$ 

\begin{lemma}
	Assume the same setting as in Definition \ref{def:assignment of a and b}. Suppose $C_1$ and $C_2$ are two components in $\widetilde{U}$ that are sent to the same component $C$ in $U$ under the  quotient map $S^3\to \rp{3}$, then \[l(C_1) = l(C_2).\]
	\label{lem:assignment of a and b}
\end{lemma}
\begin{proof}
	By applying Reidemeister moves $R$-$V$ if necessary, we can assume the local unknot $C$ intersects the essential unknot $U_1$ at exactly two points $\alpha$ and $\beta$. Then there are two possibilities according to the relative positions between $C$ and the point $p$, as shown in Figure \ref{fig:labeling}. Later, when we refer to an arc in the proof, for example arc $p\alpha$, we mean the open arc of $U_1$ traveling counterclockwise from $p$ to $\alpha$.
	
	\begin{figure}[h]
		\[{
			\fontsize{7pt}{9pt}\selectfont
			\def\svgscale{0.5}
			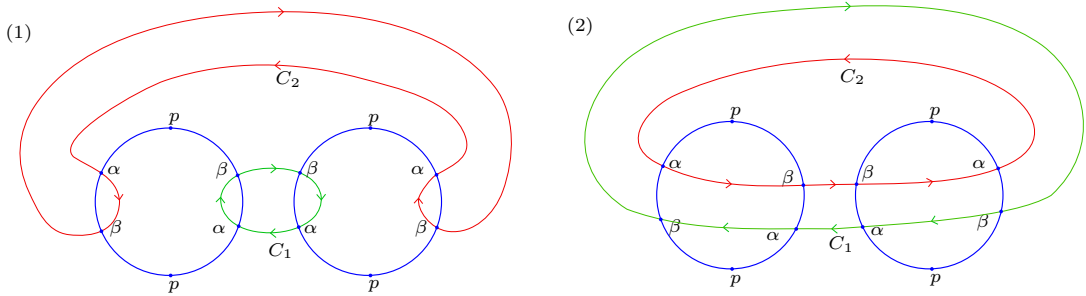
		}\]
		\caption{Two possibilities of the relative positions between $C_1, C_2$ and $p$}
		\label{fig:labeling}
	\end{figure}
	
	\begin{enumerate}
		\item In this case, the two components $C_1$ and $C_2$ are oriented in the same direction, so \[o(C_1) = o(C_2). \]
		
		The distance to infinity function $d(C_1)$ could be also be described by counting the number of intersections between $\widetilde{U}$ and the arc $p\alpha$ mod $2$, and similarly, $d(C_2)$ counts the number of intersections between $\widetilde{U}$ and the arc $\beta p$ mod $2$. By the assumption that $D$ is the diagram of an unlink, it is null homologous in particular, so the total number of intersection between $\widetilde{U}$ and the semicircle $pp$ is even. Also the link diagram $D$ has no crossing, and the circle $C_1$ bounds a disk, so the number of intersection between $\widetilde{U}$ and the arc $\alpha\beta$ should also be even. Therefore, we conclude that \[\#|\widetilde{U}\cap p\alpha|+\#|\widetilde{U}\cap \beta p| = 0 \text{ mod }2 \text{ as well},\]
	so \[d(C_1) = \#|\widetilde{U}\cap p\alpha| =\#|\widetilde{U}\cap \beta p| = d(C_2) \text{ mod }2.\]
	Then we have \[l(C_1) = l(C_2)\text{ mod }2\] in this case.
		\item This time, the orientations on $C_1$ and $C_2$ are opposite to each other, so \[o(C_1) = o(C_2) +1 \text{ mod }2.\]
		
		Again, the distance function could be described in terms of the number of intersections between $\widetilde{U}$ and various arcs. In this case, \[d(C_1) = \# |\widetilde{U}\cap p\alpha|,\,\,\,\, d(C_2) = \#|\widetilde{U} \cap p\beta| = \#|\widetilde{U} \cap p\alpha| + 1 + \#|\widetilde{U} \cap \alpha\beta|\]
		
		By the assumptions on $D$ that $D$ has no crossings and each component of $D$ is null homologous, we conclude that  \[\#|\widetilde{U} \cap \alpha\beta| =0 \text{ mod }2,\] 
		so \[d(C_1 ) = d(C_2) +1 \text{ mod }2.\]
		Therefore, we also have \[l(C_1) = l(C_2) \text{ mod }2 \] in this case. 
	\end{enumerate}
\end{proof}
With the help of the labeling function $l: \pi_0(L) \to \mathbb{Z}/2$, we can define our canonical generators associated with twisted orientations. 

\begin{definition}
Suppose $D$ is a link diagram of a null homologous link $L$ in $\rp{3}$ with a twisted orientation $o$. Pick a point $p$ on the essential unknot $U_1$ which is disjoint from the link diagram. Then we define the \textbf{canonical element $s_o$ associated with the twisted orientation $o$} in $\mathit{CBN}^{un}(D)$ as follows: first, form the twisted oriented resolution $D_o$ according to $o$, then apply the labeling function $l$ to each component of the unlink $D_o$. Finally, assign the element $a = 1+x$ to a component $C$ in $D_o$ if $l(C) = 0$, and $b =x $ to it if $l(C) =1$. 
\label{def:canonical generator}
\end{definition}

\begin{remark}
	This definition depends on the choice of $p$ in the following way: a different choice of $p$ might lead to a uniform change of $l(C)$ to $l(C)+1$ for each component $C$ in $D_o$, resulting in a switch of labels $a$ and $b$. The point $p$ plays the role of the point of infinity for the labeling function of link diagrams in $\mathbb{R}^2$.
\end{remark}

As in the case of $S^3$, these canonical generators associated with the twisted orientations give a basis for the Bar-Natan homology. First, we will show they actually lie in $\mathit{HBN}^{un} (L)$, and then we will prove that the dimensions match using an induction argument. The proof closely follows the original ones in \cite{MR2173845} for Lee homology. See also \cite{MR2286127} for a treatment of Bar-Natan homology (which is essentially the same as the one for Lee homology). The main difference in our case is that we need to take into account the $1$-$1$ bifurcation and whether there is a homologically essential component in the link. There issues are solved by the notion of twisted orientation. 

\begin{proposition}
	For a null homologous link diagram $D$ with a twisted orientation $o$, the canonical element $s_o$ represents an element in the homology $\mathit{HBN}^{un} (D)$.
	\label{prop:canonical generator}
\end{proposition}

\begin{proof}
	Following Lee's argument, we want to show that $s_o$ lies in both $ker(d)$ and $ker(d^*)$, where $d^*$ is the adjoint differential defined using the inner product on the Frobenius algebra $V$. 
		\begin{figure}[h]
		\[{
			\fontsize{7pt}{9pt}\selectfont
			\def\svgscale{0.6}
			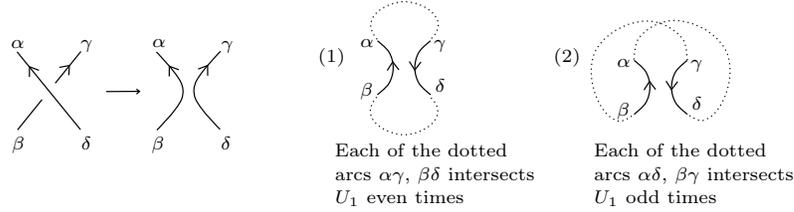
		}\]
		\caption{ Two arcs sharing a crossing in the twisted oriented resolution won't belong to the same circle }
		\label{fig:local_orientation}
	\end{figure}
	
	Let's look more carefully at what happens near a crossing. Without loss of generality, suppose the local picture near a crossing looks like the one in Figure \ref{fig:local_orientation}. As in the usual proof for links in $S^3$, we want to show the two arcs $\alpha \beta$ and $\gamma \delta$ won't belong to the same circle in the resolution. In the usual proof, this possibility is ruled out because if these two arcs belong to the same circle, then the endpoint $\alpha$ is adjacent to $\gamma$ and $\beta$ is adjacent to $\delta$ in the circle due to the absence of crossings in the resolution. However, this arrangement is incompatible with the orientation on the arc $\alpha\beta$ and $\gamma\delta$. 
	
	In our case, the situation is slightly different because the projection lies in $\rp{2}$ instead of $\mathbb{R}^2$. We again prove by contradiction that the arcs $\alpha\beta $ and $\gamma\delta$ won't belong to the same circle in the twisted oriented resolution. Suppose they were, then there are two possibilities:
	
\begin{enumerate}
	\item  If the endpoint $\alpha$ is adjacent to $\gamma$ on the circle, then the arc $\alpha \gamma$ of the circle intersects the essential unknot $U_1$ an even number of times. Therefore, the twisted orientation on the circle will be from $\beta$ to $\alpha$ and from $\gamma$ to $\delta$, or from $\alpha$ to $\beta$ and from $\delta$ to $\gamma$, which is incompatible with the local orientation near the crossing.
	\item If the endpoint $\alpha$ is adjacent to $\delta$ on the circle, then the arc $\alpha \delta$ of the circle intersects the essential unknot $U_1$ an odd number of times. Therefore, the twisted orientation on the circle will be again from $\beta$ to $\alpha$ and from $\gamma$ to $\delta$, or from $\alpha$ to $\beta$ and from $\delta$ to $\gamma$, as the arrow is switched an odd number of times in the twisted orientation. But this is again incompatible with the local orientation near the crossing.
\end{enumerate}

Therefore, we conclude that the arcs $\alpha\beta$ and $\gamma\delta$ near a crossing won't belong to the same circle in the twisted oriented resolution. This implies if we change the smoothing at crossing in the twisted orientation resolution, it won't be a $1$-$1$ bifurcation or a $1$-$2$ bifurcation, which is exactly what we want to avoid, as the corresponding maps are non-trivial. The rest of the proof is exactly the same as in the usual case for links in $S^3$, as the rule we used to assign $a$ and $b$ to each circle in the resolution guarantees that the label assigned to the circle containing the arc $\alpha\beta$ is different from that assigned to the circle containing the arc $\gamma \delta $ in the resolution. Therefore, $s_o$ lies in $ker(d) \cap ker(d^*)$.
	
\end{proof}

Now we prove the dimension of $\mathit{HBN}^{un}(D)$ using an induction argument on the number of crossings. There is some adjustment we need to make compared to the original one in \cite{MR2173845}, as we need to discuss whether there is a homologically essential component in the link or not.

\begin{proposition}
	Suppose $L$ is a null homologous link in $\rp{3}$ with a link diagram $D$ in $\rp{2}$. Then 
	\begin{equation*}
		dim(\mathit{HBN}^{un}(D)) =
		\begin{cases}
			0 & \text{if $L$ has some component that is non-zero in $H_1(\rp{3},\mathbb{Z})$;}\\
			2^{|L|} & \text{otherwise, $i.e.,$ all components of $L$ are null homologous.}\\
		\end{cases}       
	\end{equation*}

What's more, a basis of $\mathit{HBN}_{un}(D)$ is given by \[\left\{s_o \mid o \text{ is a twisted orientation of } L\right\}.\]
\end{proposition}
\begin{remark}
	If $L$ has some component which is non-zero in $H_1(\rp{3},\mathbb{Z})$, then there is no twisted orientation on $L$, so the set of basis given by twisted orientations is empty, and $\mathit{HBN}^{un} (L)$ is $0$-dimensional in this case.
\end{remark}

\begin{proof}
	 Following Lee's original proof, we first prove it for knots and $2$-component links by induction on the number of crossings. It clearly holds for unknot and unlink with $2$ local unknot components. 
	 		\begin{figure}[h]
	 	\[{
	 		\fontsize{7pt}{9pt}\selectfont
	 		\def\svgscale{0.6}
	 		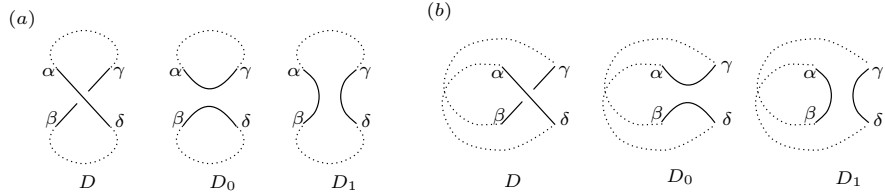
	 	}\]
	 	\caption{ Resolve a null homologous knot/$2$-component link at a crossing}
	 	\label{fig:induction_1}
	 \end{figure}
 
	 Now let $D$ be a link diagram of a null homologous knot with $n$ crossings. Pick one crossing, and let $D_0$, $D_1$ denote the $0$- and $1$-smoothing of $D$ at this crossing respectively. Refer to $(a)$ in Figure \ref{fig:induction_1} as a local model near the crossing. Suppose, without loss of generality, that the endpoint $\alpha$ is adjacent to the endpoint $\gamma$ on the knot $K$. (The case $\alpha $ is adjacent to the endpoint $\beta$ is the same, by changing the over-crossing to an under-crossing, which switches $D_0$ and $D_1$ but doesn't change the argument otherwise. The endpoint $\alpha$ won't be adjacent to $\delta$, as that will give a $2$-component link instead of a knot.) Then we divide into two cases depending on how many times the arc $\alpha\gamma $ of $K$ intersects the essential unknot $U_1$.
	 \begin{enumerate}
	 	\item If the segment $\alpha \gamma$ intersects $U_1$ an odd number of times, so does the segment $\beta \delta$, as we assume $K$ is null homologous. Then the link represented by $D_0$ consists of two homologically essential components, so by the induction hypothesis, it has trivial Bar-Natan homology. The diagram $D_1$ represents a null homologous knot with one less crossing than $D$, so by the induction hypothesis, we have \[dim(\mathit{HBN}^{un} (D_1))=2.\] Using the long exact sequence relating $\mathit{HBN}^{un} (D),\mathit{HBN}^{un} (D_0),\mathit{HBN}^{un} (D_1)$, we conclude that 
	 	\[dim(\mathit{HBN}^{un} (D))=2.\]
	 	\item If the segment $\alpha\gamma$ intersects $U_1$ an even number of times, so does the segment $\beta\delta$. Then $D_0$ consists of $2$ null homologous components and $D_1$ is a null homologous knot, so by the induction hypothesis, \[dim(\mathit{HBN}^{un} (D_0)=4, \,\,\, dim (\mathit{HBN}^{un} (D_1))=2.\] As in the usual case, out of the four twisted orientations on $D_0$, there are two which are compatible under the change of smoothings to twisted orientations on $D_1$, and the map in the long exact sequence will send these two twisted oriented generators in $  \mathit{HBN}^{un} (D_0)$ to the corresponding ones in $\mathit{HBN}^{un} (D_1)$. Therefore, \[2\leq dim(\mathit{HBN}^{un} (D))\leq dim(\mathit{HBN}^{un} (D_0))+dim(\mathit{HBN}^{un} (D_1))-4=2.\]
	 \end{enumerate}
 
 Now suppose $D$ is a diagram of a null homologous link $L$ of two components, $K_0$ and $K_1$. Suppose first that there are no crossings between the two components in the link diagram $D$. Then $K_0$, $K_1$ both need to be null homologous (otherwise, they are both homologically essential, and there has to be at least one crossing between these two components), and we can apply the K\"unneth formula to conclude \[dim(\mathit{HBN}^{un} (D))=4.\]
 
 Suppose the two components share at least one crossing. Pick any crossing shared by them and form the $0$- and $1$-smoothing, $D_0$ and $D_1$, respectively. Refer to $(b)$ in Figure \ref{fig:induction_1} as a local model. We again divide into two cases depending on whether $K_0$ and $K_1$ are null homologous or not.
 \begin{enumerate}
 	
 	\item If $K_0,K_1$ are both null homologous, then each of the segments $\alpha\delta$ and $\beta\gamma$ intersects the essential unknot $U_1$ an even number of times, and the two twisted orientations on $D_0$ are incompatible with the two twisted orientations on $D_1$ under changing the smoothing. So the map from $\mathit{HBN}^{un} (D_0)$ to $\mathit{HBN}^{un} (D_1)$ is $0$ in the long exact sequence, and \[4\leq dim(\mathit{HBN}^{un} (D))\leq dim(\mathit{HBN}^{un} (D_0))+dim(\mathit{HBN}^{un} (D_1))=4.\]
 	\item If $K_0$, $K_1$ are both homologically essential, then each of the segments $\alpha\delta$ and $\beta\gamma$ intersects the essential unknot $U_1$ an odd number of times, and the two twisted orientations on $D_0$ are compatible with the two twisted orientations on $D_1$ under changing the smoothing. So the map in the long exact sequence sends $\mathit{HBN}^{un} (D_0)$ isomorphically to $\mathit{HBN}^{un} (D_1)$, and therefore \[dim (\mathit{HBN}^{un} (D))=0.\]
 	
 \end{enumerate}

This finishes the discussion of knots and $2$-component links. For links with more components, if there is a component which doesn't share a crossing with any other components, then we can apply K\"unneth formula.
Otherwise, there is at least one crossing shared by different components.  Again, we divide into two cases depending on whether there exists homological essential component or not.

\begin{enumerate}
	\item Suppose all components are null homologous, then we apply the same argument as in the case $(a)$ for $2$-component links.
	\item Suppose there are some homological essential components. Then, there has to be at least two such components, and they must share a crossing in the link diagram. Choose one such crossing, and then we apply the same argument as in case $(b)$ for $2$-component links.
\end{enumerate}
\end{proof}

It is time to do the global grading shift to ensure the Bar-Natan homology is a link invariant in $\rp{3}$ as a bigraded vector space. The usual convention is to apply a shift of $-n_-$ in the homological grading and a shift of $n_++2n_-$ in the quantum grading, where $n_+$ and $n_-$ are number of positive and negative crossings for an oriented link, respectively. As expected, we will get an invariant for twisted oriented links in $\rp{3}$, and we should use $n_+$ and $n_-$ counted with respect to the twisted orientation. 

	\begin{definition}
Let $D$ be a link diagram in $\rp{2}$ of a null homologous link in $\rp{3}$ with a twisted orientation. Let $n_+$ and $n_-$ be the number of positive and negative crossings, respectively, counted with respect to the twisted orientation on $D$. Then the \textbf{Bar-Natan chain complex } $\mathit{CBN}_{*,*}(D)$ is defined as \[\mathit{CBN}_{*,*}(D) = \mathit{CBN}^{un}_{*,*}(D)\left\{-n_-\right\}[n_+-2n_-],\] where $\mathit{CBN}^{un}_{*,*}(D)$ is the unadjusted Bar-Natan chain complex defined in $\ref{def:unadjusted chain complex}$, $\left\{-n_-\right\}$ means shifting down by $n_-$ in the homological grading, and $\left[n_+-2n_-\right]$ means shifting up by $n_+-2n_-$ in the quantum grading. The homology $\mathit{HBN}_{*,*}(D)$ of $\mathit{CBN}_{*,*}(D)$ is called the \textbf{Bar-Natan homology} of null homologous links in $\rp{3}$ with a twisted orientation.

\label{def:Bar-Natan complex}
\end{definition}

We are going to prove $\mathit{HBN}_{*,*}$ is an invariant of twisted oriented null homologous links in $\rp{3}$ by exhibiting filtered maps between $\mathit{CBN}_{*,*}$ induced by Reidemeister moves, which are isomorphisms on $\mathit{HBN}_{*,*}$. Furthermore, we are going to prove that these filtered map send canonical generators associated with the twisted orientation to the corresponding canonical generators. This follows the same strategy as in Section 6 in \cite{MR2729272}. See also Section 6 in \cite{MR2286127} for the Reidemeister moves in the usual Bar-Natan homology. We need to check a few more cases because the projection lies in $\rp{2}$, but there is no extra difficulty.

\begin{proposition}
	\label{prop:R moves}
	For each Reidemeister move as drawn in Figure \ref{fig:r-moves} relating link projections $D_0$ to $D_1$, we can define a filtered chain map \[\rho': \mathit{CBN}_{*,*}(D_0) \to \mathit{CBN}_{*,*}(D_1),\] such that the grading-preserving part $\rho$ of $\rho'$ gives an isomorphism on the Khovanov homology for links in $\rp{3}$, and \[\rho'(\left[s_o\right])=\left[s_{o'}\right],\] where $o$ is a twisted orientation on $D_0$ and $o'$ is the induced twisted orientation on $D_1$.
\end{proposition}

\begin{proof}
	There are three things to do in the proof: First, we need to give the definition of the filtered chain map $\rho'$; Second, we need to check the grading-preserving part $\rho$ agrees with the one used in \cite{MR3189291} for Reidemeister moves on Khovanov homology in $\rp{3}$; Third, we need to check whether $\rho'$ sends canonical generators to canonical generators. As mentioned, the chain map $\rho'$ we are going to use will be the same as the ones in \cite{MR2729272}, and the prof is purely a book-keeping check. The situations for the Reidemeister moves I,IV and V are either trivial or identical to the case in $S^3$. However, for the Reidemeister moves
	II and III, they will involve some more case-by-case analysis than the usual situation in $S^3$. We will do some examples and leave the rest to the reader.
	
	For the Reidemeister moves IV and V, they don't change the chain complex at all, and we can take $\rho'$ to be the identity. It is worth mentioning that for the Reidemeister move V, as we move an arc through the essential unknot $U_1$, the orientations on this arc in $o$ and $o'$ will be reversed to each other.
	
	For the Reidemeister move I, since the operation of adding a curl is local, and the chain maps only involve $1$-$2$ and $2$-$1$ bifurcations, the same map and argument as in \cite{MR2729272} works, with $a = 1+x$, $b=x$. In terms of the global grading shift, if we add a positive curl, then we also add $1$ to $n_+$ for our definition of $n_+$. Therefore, the usual check of global grading shifts works in the same way.
	
			 		\begin{figure}[h]
		\[{
			\fontsize{7pt}{9pt}\selectfont
			\def\svgscale{0.6}
			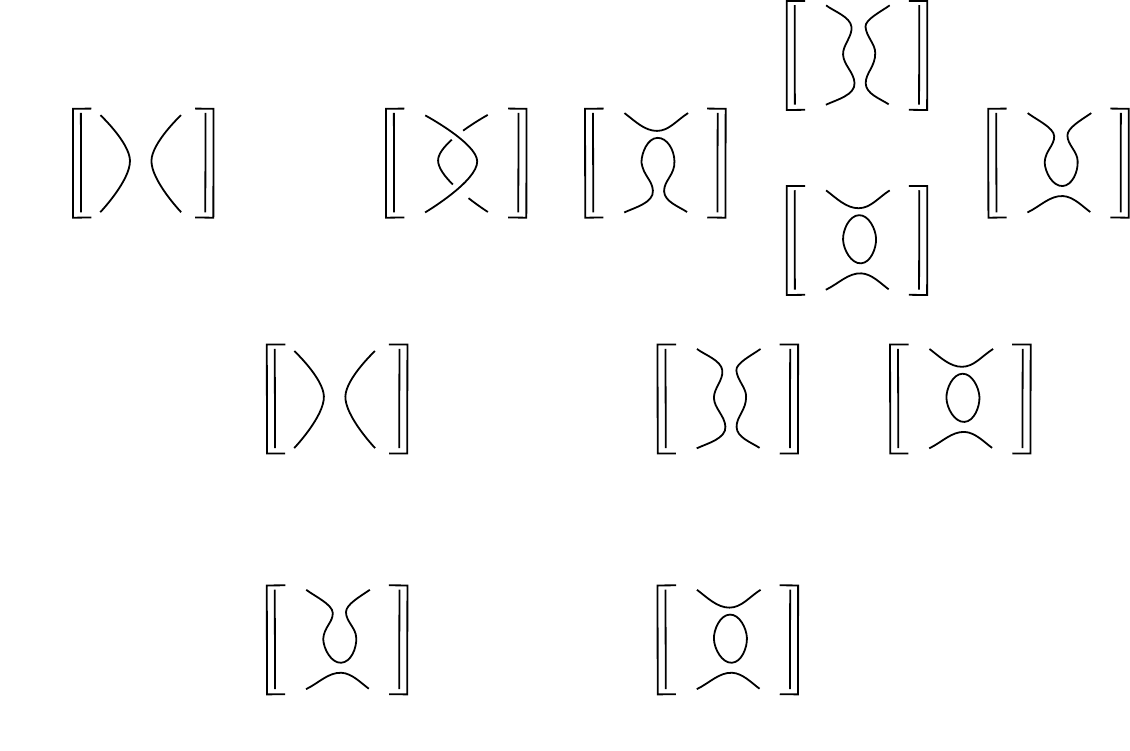
		}\]
		\caption{ Definition of $\rho'$ for $R$-$II$ moves }
		\label{fig:r2}
	\end{figure}
	For the Reidemeister move II, refer to Figure \ref{fig:r2} for the definition of the map $\rho'$. Here, $\iota$ is the map sending the state $z$ to $z\otimes 1$, where $1$ is assigned to the extra circle in the middle. In the chain complex $\mathit{CBN}(D_1)$, we know the maps $\Delta$ and $m$ because they correspond to a local splitting/merging of a circle. However, we don't know the maps $d_1'$ and $d_2'$: they could be any of $\Delta, m$ or $f$ depending on how the rest of the link diagram looks like. But the point is we don't need to use their properties in the definition of $\rho'$.  The grading-preserving part $\rho$ of $\rho'$ gives an isomorphism on Khovanov homology in $\rp{3}$, using the usual proof of canceling acyclic complexes (note that we only the property of $\Delta $ and $m$ in the proof for invariance of Khovanov homology in $\rp{3}$).
	
	To check that $\rho'$ sends canonical generators to canonical generators, we need to divide into further cases, depending on how the two arcs are connected and oriented in the link diagram.
	
	\begin{enumerate}
\item If the two arcs are oriented in the same direction, then by the argument in Proposition \ref{prop:canonical generator}, the two arcs won't belong to the same circle in the twisted oriented resolution, and the labels on these two circles in $s_o$ are different. Then $d_2'$ is a $2$-$1$ bifurcation map, which is given by $m$, and $m(s_o)=0$, as it merges two circles with different labels. Therefore, \[\rho'(s_o) = (s_{o}, \iota\circ d_2'(s_o)) = (s_o,0) = s_{o'}.\] 
\item If the two arcs are oriented in opposite directions, then we divide into cases depending on whether these two arcs belong to the same circle or not in the twisted oriented resolution.
\begin{enumerate}
	\item 
		\begin{figure}[h]
		\[{
			\fontsize{8pt}{10pt}\selectfont
			\def\svgscale{0.6}
			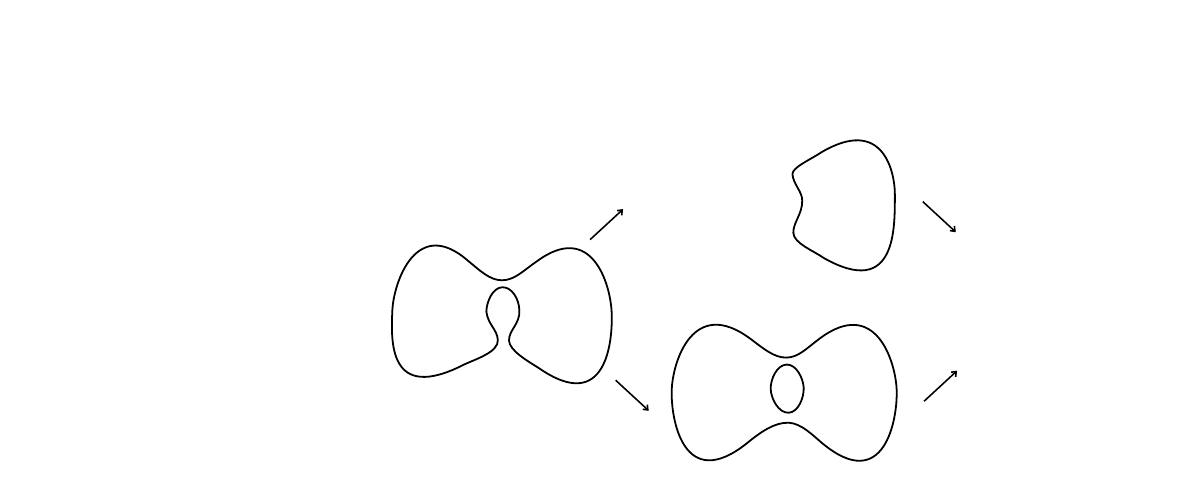
		}\]
		\caption{Schematic drawing for case $(2)(a)$ of $R$-$II$ moves}
		\label{fig:r2_1}
	\end{figure}
	
	If the two arcs belong to two different circles in the twisted oriented resolution, then the labels on the two circle will be the same in $s_o$, say both of them are $a$, following the rules we assign labels in Definition \ref{def:canonical generator}. See Figure \ref{fig:r2_1} for an illustration. Then $d_1'$ is a $1$-$2$ bifurcation map, and $d_2'$ is a $2$-$1$ bifurcation map. Denote $s_o$ by $a\otimes a$, which is the label of $s_o$ on the two circle that are changing through the Reidemeister move II. Then \[\iota\circ d_2'(a\otimes a) = \iota(a) = a\otimes 1 = a\otimes a + a\otimes b,\] and \[\rho'(a\otimes a) = (a\otimes a, a\otimes a +a\otimes b).\] Note that \[da = (d_1'(a),\Delta (a)) = (a\otimes a,a\otimes a),\] so \[\left[\rho'(s_o)\right] = \left[(0,a\otimes b) \right]= \left[s_{o'}\right] \text{ in } \mathit{HBN}(D_1).\]

	\item If the two arcs belong to the same circle in the twisted oriented resolution, then again we need to divide into two cases, depending on whether a $1$-$1$ bifurcation is involved or not.
	\begin{enumerate}
		\item 
		
				\begin{figure}[h]
			\[{
				\fontsize{8pt}{10pt}\selectfont
				\def\svgscale{0.6}
				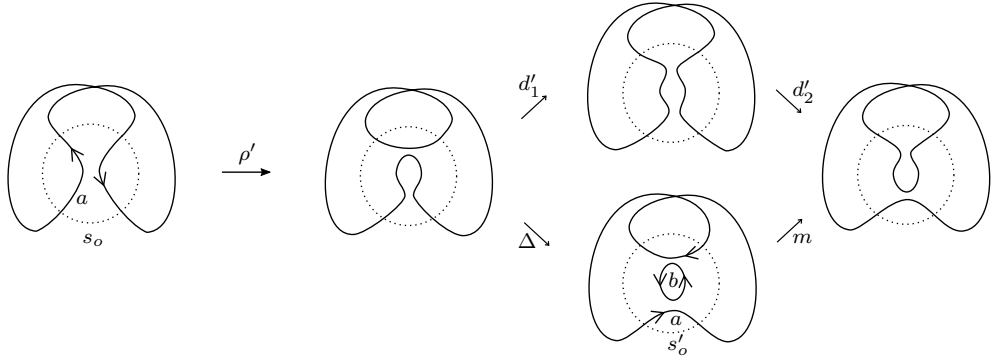
			}\]
			\caption{Schematic drawing for case $(2)(b)(i)$ of $R$-$II$ moves}
			\label{fig:r2_2}
		\end{figure}

		Suppose both $d_1'$ and $d_2'$ are $1$-$1$ bifurcations. Let's assume the label on the this circle in $s_o$ is $a$. See Figure \ref{fig:r2_2} for an illustration. Following the similar notation, we denote $s_o$ by $a$, and \[\iota\circ d_2'(a) = \iota(a) = a\otimes 1 = a\otimes a + a\otimes b,\] so \[\rho'(a) = (a,\iota\circ d_2'(a)) = (a,a\otimes a +a\otimes b). \] Again, we can cancel $(a,a\otimes a) $ in homology as \[da = (d_1'(a),\Delta(a)) = (a,a\otimes a),\] so \[\left[\rho'(s_o)\right] = \left[(0,a\otimes b)\right] = \left[s_{o'}\right] \text{ in } \mathit{HBN}(D_1). \]
		\item Suppose that none of $d_1'$ and $d_2'$ are $1$-$1$ bifurcation. Then it is the same as the usual check for the Reidemeister move II as in $S^3$, and we leave it to the reader. 
	\end{enumerate}
    \end{enumerate}
	\end{enumerate}

For the Reidemeister move III, the definition of the map $\rho'$ is again similar to the corresponding one for the Reidemeister move III in $S^3$, for the same reason as mentioned above: the edge maps that we need some properties of to define $\rho'$ are local, $i.e.$ they only involve the splitting/merging of a local circle. The proof that the grading-preserving part of $\rho'$ induces an isomorphism on Khovanov homology follows the same approach as in the case of $S^3$, as the proof of invariance of Khovanov homology under the Reidemeister move III in $\rp{3}$ is carried out similarly to that in $S^3$. 

To verify that it sends canonical generators to canonical generators, we need to consider different cases depending on how each strand is oriented and how these strands are connected outside the local region in the twisted oriented resolution. This requires more case analysis than the proof in $S^3$ because the link diagram lies in $\rp{2}$. We will illustrate one example and leave the rest to the reader.

\begin{figure}[h]
	\[{
		\fontsize{8pt}{10pt}\selectfont
		\def\svgscale{0.6}
		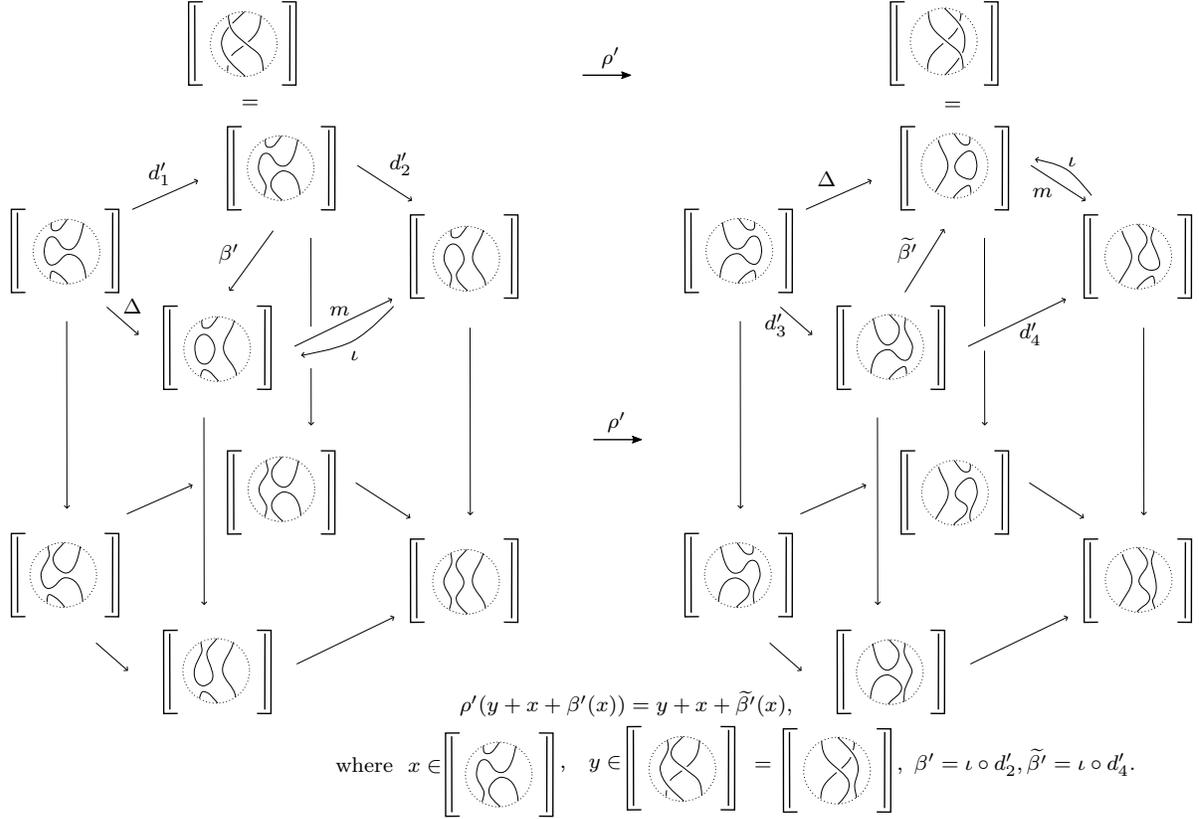
	}\]
	\caption{Definition of $\rho'$ for $R$-$III$ moves}
	\label{fig:definition_for_r3}
\end{figure}

The definition of the map $\rho':\mathit{CBN}(D_0) \to \mathit{CBN}(D_1)$ is the same as that in \cite{MR2729272}, which is recalled in Figure \ref{fig:definition_for_r3}. The grading-preserving part of $\rho'$ gives the induced map by the Reidemeister move III on the Khovanov homology in $\rp{3}$, as discussed above, so it remains to show \[\left[\rho(s_o)\right] = \left[s_{o'}\right] \text{ in } \mathit{HBN}(D_1).\] 

As in the proof in \cite{MR2729272}, three out of the four situations regarding relative orientations lead to trivial checks. For the remaining relative orientation, there are several more cases to examine than in \cite{MR2729272}, as the diagram lies in $\rp{2}$. In these additional cases, some of $d_2'$ and $d_4'$ will be the $1$-$1$ bifurcation map. As an example, we consider the following orientation on the strands and connectivity in the twisted oriented resolution, as shown in Figure \ref{fig:r3}. Note that the rule that we use to assign labels $a$ or $b$ to each circle depends on the choice of the point $p$. However, changing the position of $p$ results in a uniform switch between $a$ and $b$, so it won't affect the argument. Here, we demonstrate one possibility of the labeling for a specific choice of $p$ lying in the indicated region in the diagram.

In this case, $d_1'$ and $d_2'$ are $1$-$1$ bifurcations, $d_3'$ is a $1$-$2$ bifurcation, and $d_4'$ is a $2$-$1$ bifurcation. Note that \[\left[s_o\right] = [x+\beta'(x)],\,\,\,\,\text{ and } \left[s_{o'}\right]=\left[x+\widetilde{\beta}'(x)\right],\] where $x$ is as shown in Figure \ref{fig:r3}. Then \[\rho'(\left[s_o\right]) = \rho'(\left[x+\beta(x)\right]) = \left[x+\widetilde{\beta}'(x)\right] = \left[s_{o'}\right] \text{ in } \mathit{HBN}(D_1).\] 

		 		\begin{figure}[h]
	\[{
		\fontsize{7pt}{9pt}\selectfont
		\def\svgscale{0.6}
		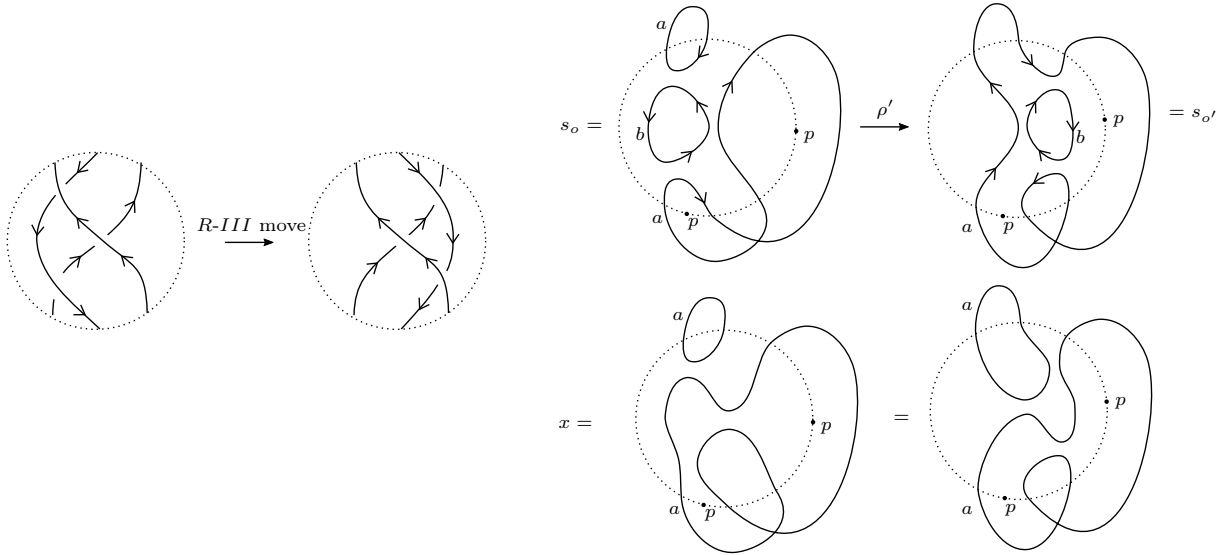
	}\]
	\caption{ An example computation of $\rho'$ for a $R$-$III$ move, with given orientation and connectivity }
	\label{fig:r3}
\end{figure}

\end{proof}

Combining all the discussion in this section, we get a well-defined Bar-Natan homology for twisted oriented null homologous links in $\rp{3}$.

\begin{theorem}
	The Bar-Natan homology $\mathit{HBN}(L)$ is a twisted oriented link invariant for null homologous links in $\rp{3}$ as a bigraded vector space, with a canonical basis \[\left\{s_o \mid o \text{ is a twisted orientation of } L\right\}.\] 
\end{theorem}

\begin{remark}
	When $L$ is a null homologous knot, the usual proof for knot in $S^3$ showing that $s_o$ lies in homological grading $0$ works here as well, as we use the twisted orientation to define $n_+,n_-$, as well as to form the resolution. For links, the homological grading of $s_o$ could be read off from the linking number between components of the covering link $\widetilde{L}$ in $S^3$, with the orientation on $\widetilde{L}$ that lifts the twisted orientation.
\end{remark}
	\section{Cobordism maps on Bar-Natan homology in $\rp{3}$}
	\label{sec:cobordism}
	In the usual setting for Bar-Natan or Lee homology in $S^3$, the reason one can get some genus bound from the homology is that one can associate some filtered chain maps to oriented cobordisms between links, which interact well with the canonical generators of the homology given by the orientations, and the filtration degree of the map is bounded by the Euler characteristic of the cobordism. The rule of thumb of this paper is to replace 'orientation' with 'twisted orientation' everywhere, so we need to develop a notion of twisted orientation on cobordisms. The nature definition of twisted orientable cobordism should ensure that it carries the twisted orientation from one end of the cobordism to the other end; that is, when the cobordism is formed by attaching bands, it should be done in a way that is compatible with the twisted orientation. More formally, we provide the following definition from the perspective of the double covers in $S^3$.
	
	\begin{definition}
		Let $\Sigma:L\to L'$ be a cobordism between null homologous links $L$ and $L'$ in $\rp{3}\times I$. A \textbf{twisted orientation} on $\Sigma$ is an orientation on its double cover $\widetilde{\Sigma}: \widetilde{L}\to \widetilde{L'}$ in $S^3\times I$, which is reversed by the deck transformation $\tau:S^3\times I \to S^3\times I$ in the $S^3$-direction. The cobordism $\Sigma$ is \textbf{twisted orientable} if such a twisted orientation exists.
		\label{def:twisted orientable cobordism}
	\end{definition}

	In particular, the double cover $\widetilde{\Sigma}$ is necessarily orientable if $\Sigma$ is twisted orientable. This definition is natural in view of the definition of twisted orientation on null homologous links in Definition \ref{def:twisted orientation on links}.

	Note that a twisted orientable cobordism in $\rp{3}\times I$ could be either orientable or unorientable in the usual sense, and there are orientable/unorientable cobordisms which are not twisted orientable, as demonstrated in the next example.
	
	
		 		\begin{figure}[h]
		\[{
			\fontsize{7pt}{9pt}\selectfont
			\def\svgscale{0.6}
			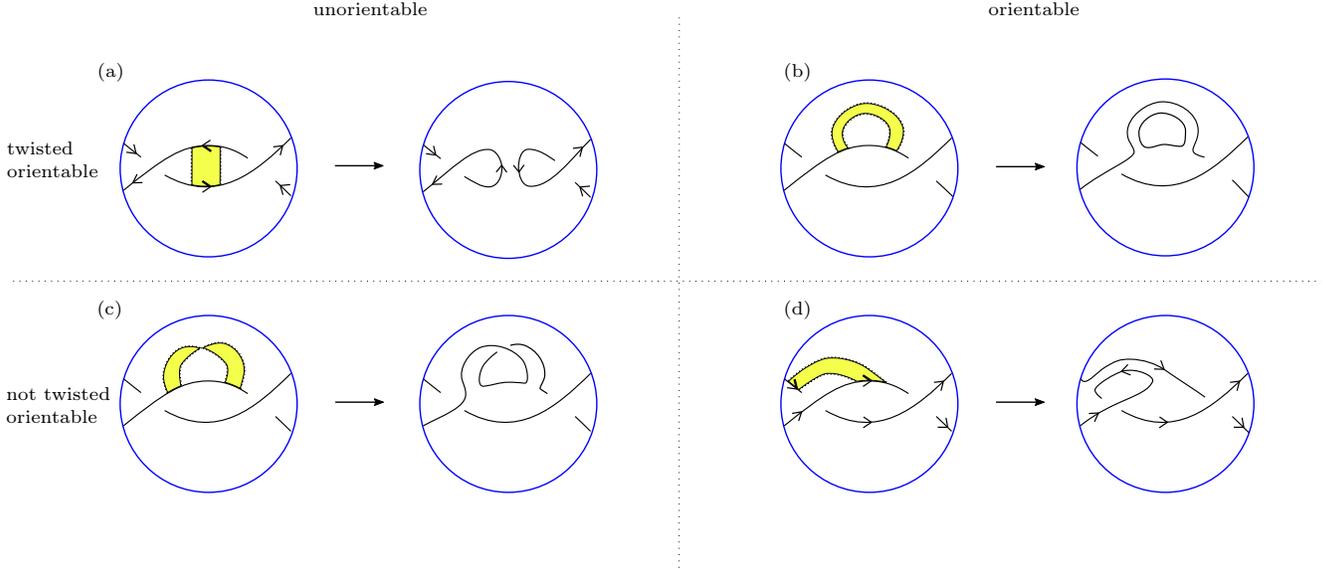
		}\]
		\caption{Different possibilities of twisted-orientability and usual orientability of a band attachment}
		\label{fig:twisted_orientable_cobordism}
	\end{figure}
	
	\begin{example}
		\label{ex:twisted orientable}
		 Figure \ref{fig:twisted_orientable_cobordism} shows all $4$ possibilities of surface cobordism in $\rp{3} \times I$, depending on whether it is twisted orientable and orientable in the usual sense.

		 The cobordism in (a) is twisted orientable, where the arrows represent twisted orientations on the boundary knots, and the band is added in an orientation-compatible way with respect to the twisted orientations. However, this cobordism is unorientable in the usual sense: it is topologically a M\"{o}bius band with a disk removed. 
		 
		 The cobordism in (b) is both twisted orientable and orientable in the usual sense. 
		 
		 The cobordism in (c) is not twisted orientable and unorientable in the usual sense.
		 
		 The cobordism drawn in (d) is not twisted orientable, where the arrows represent usual orientations on the boundary knots. The band is added in an orientation-compatible way with respect to the usual orientations, making this cobordism orientable in the usual sense. 
	\end{example}

	Now we are going to prove an analogous statement about the interaction between twisted orientable cobordisms and the canonical generators of the Bar-Natan homology generated by twisted orientation. The idea is to decompose $\Sigma$ into elementary pieces consisting of those corresponding to Reidemeister moves and a single handle attachment. To do that, we first prove that the twisted orientation on $\Sigma$ descends to a twisted orientation on each elementary pieces and their input and output boundaries.
	
	\begin{lemma}
		Suppose $\Sigma:L\to L'$ is a twisted orientable cobordism between homologous links $L$ and $L'$ in $\rp{3}\times I$. Decompose \[\Sigma=\Sigma_1\circ\Sigma_2\circ...\circ\Sigma_n,\] where each $\Sigma_i$ corresponds to either a Reidemeister move or a single handle attachment. Let $L_i$ denote the output boundary of the composition $\Sigma_1\circ\Sigma_2\circ...\circ\Sigma_i$, with $L_0 = L$ and $L_n = L'$. Then a twisted orientation on $\Sigma$ restricts to a twisted orientation on each of the $\Sigma_i$'s and $L_i$'s.
	\end{lemma}

\begin{proof}
	Consider the double cover of the decomposition \[\widetilde{\Sigma}=\widetilde{\Sigma_1}\circ\widetilde{\Sigma_2}\circ...\circ\widehat{\Sigma_n},\] where each $\widetilde{\Sigma_i}$ corresponds to performing a pair of equivariant Reidemeister moves, or attaching a pair of equivariant handles. A twisted orientation on $\Sigma$ is an orientation on $\widetilde{\Sigma}$ which is reversed under the action of the deck transformation $\tau$ in the $S^3$-direction. Since the action of $\tau$ is fiberwise in the $S^3$ direction of $S^3\times I$, such an orientation on $\widetilde{\Sigma}$ restricts to an orientation on each $\widetilde{\Sigma_i}$ which is also reversed by $\tau$, so the twisted orientation on $\Sigma$ restricts to each $\Sigma_i$. 
	
	As for the links $L_i$, we again look at the double cover $\widetilde{L_i}$. Since they are regular fibers of the projection from $\widetilde{\Sigma}$ to $I$, a tubular neighborhood of $\widetilde{L_i}$ in $\widetilde{\Sigma}$ is homeomorphic to $\widetilde{L_i}\times \left[-\epsilon,\epsilon\right]$, where $\tau$ acts trivially in the $\left[-\epsilon,\epsilon\right]$-direction, as $\tau$ is a fiberwise action. However, $\tau$ acts in an orientation-reversing way on the neighborhood $\widetilde{L_i}\times \left[-\epsilon,\epsilon\right]$, so it must reverse the orientation on $\widetilde{L_i}$. Hence, we obtain a twisted orientation on $L_i$.
\end{proof}
	\begin{remark}
		In particular, the above lemma implies that every component of $L_i$ is null homologous, as $L_i$ is twisted orientable. Hence $\mathit{HBN}(L_i)$ is non-trivial for each $i$. Note that it is not true that all closed loops in $\Sigma$ are null homologous; for example, the fiber over the critical value of a $1$-handle attachment has a tubular neighborhood homeomorphic to a M\"obius band, as could be seen in (a) of Figure \ref{fig:twisted_orientable_cobordism}. In that case, the action of $\tau$ on a tubular neighborhood $\widetilde{L}\times \left[-\epsilon,\epsilon\right]$ is the usual covering map from an annulus to a M\"obius band, which is non-trivial in the $\left[-\epsilon,\epsilon\right]$-direction and orientation-preserving when restricted to $\widetilde{L}$.
	\end{remark}

	Now we can prove the analogous statement about the effect of the maps induced by twisted orientable cobordisms on Bar-Natan homology in $\rp{3}$. 
	
	\begin{proposition}
		\label{prop:bound on Euler characteristics}
		Suppose $\Sigma:L\to L'$ is a twisted orientable cobordism between null homologous links $L$ and $L'$ in $\rp{3} \times I$. Then one can define a filtered chain map of degree $\chi(\Sigma)$, \[F_{\Sigma}:\mathit{CBN}_{*,*}(L) \to \mathit{CBN}_{*,*}(L'),\] such that
		\begin{equation*}
			F_{\Sigma}(\left[s_o\right]) = \sum_{\left\{o_i\right\}}\left[s_{o_{i_{|L'}}}\right], \tag{$*$}
		\end{equation*}
		
		where $o$ is a twisted orientation on $L$, $\left\{o_i\right\}$ is the set of twisted orientations on $\Sigma$ that restrict to $o$ on $L$, $o_{i_{|L'}}$ is the restriction of such twisted orientations on $L'$, and $s_o$ (respectively $s_{o_{I_{|L'}}}$) are the corresponding canonical generators of $\mathit{HBN}(L)$ (respectively $\mathit{HBN}(L')$) defined as in Definition \ref{def:canonical generator}.
	\end{proposition}

\begin{proof}
	As in the proof of analogous statement in \cite{MR2729272}, we divide the cobordism $\Sigma$ into elementary pieces, define the map for each elementary piece, and check the desired properties hold on each piece. The above lemma proves that each elementary piece of $\Sigma$ is twisted orientable between twisted orientable links. See also Theorem 5.5 in \cite{manolescu2023rasmussen} for the similar statement in the situation of Lee homology in $\rp{3}$.

	Recall Definition \ref{def:assignment of a and b} where we describe the rules to assign $a$ and $b$ to each circle in the twisted oriented resolution for the definition of $s_o$, we need to choose a point $p$ on the essential unknot $U_1$. Similarly here, we will choose a point $p$, such that $p\times I \subset \rp{3} \times I$ is away from the cobordism $\Sigma$. This chosen $p$ will serve as the reference point when talking about the canonical generators in $\mathit{HBN}(L_i)$ for each $L_i$.

	The cobordism $\Sigma $ could be expressed as the composition of a sequence of Reidemeister cobordisms and elementary Morse cobordisms. For the Reidemeister cobordisms, we will use the map $\rho'$ defined in Proposition \ref{prop:R moves}. Each of these Reidemeister cobordisms is topologically a cylinder of Euler characteristic $0$, and there is a unique twisted orientation on it that extends the given twisted orientation on the input. Proposition \ref{prop:R moves} shows these maps satisfy the equation $(*)$. 
	
	Elementary Morse cobordisms are given by attaching $0$, $1$ or $2$-handles. For the $0$ and $2$-handle attachments, which are local, we use the unit and counit maps:\[\iota: \mathbb{F} \to V, \,\,\,\, \eta:V \to \mathbb{F}\] in the Frobenius algebra structure of $V$ as defined in Definition \ref{def:frobenius algebra}. The twisted orientability condition puts no further restrictions on $0$ and $2$-handle attachments, as they correspond to attaching a pair of disjoint $0$ and $2$-handles in the double cover where $\tau$ acts by switching the two copies. Such cobordisms are always twisted orientable, where a twisted orientation is a pair of opposite orientations on the two copies of the handles in the double cover. It is trivial to check that the desired properties hold for these cobordism maps, as in the usual case for cobordism in $S^3\times I$.
	
	For the $1$-handle attachments, we will use $m,
	\Delta$ and $f$ on each smoothing of the cube of resolutions, depending on whether it is $2$-$1$, $1$-$2$, or $1$-$1$ bifurcation in each resolution. Each of $m,\Delta$ and $f$ is a filtered map of degree $-1$, and the proof that such a definition gives a chain map is the same as checking $d^2=0$ in the Bar-Natan chain complex, which is covered in Lemma \ref{lem:1-1bifurcation} and the discussion after that. 
		 		\begin{figure}[h]
		\[{
			\fontsize{7pt}{9pt}\selectfont
			\def\svgscale{1}
			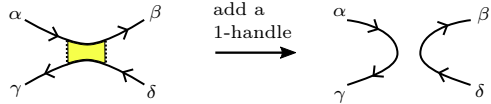
		}\]
		\caption{Local picture near a $1$-handle attachment}
		\label{fig:local_band_addition}
	\end{figure}

	It is left to show that the canonical generators associated with twisted orientation behave nicely under such maps. Again, the reason it works is that we attach $1$-handles which are compatible with the twisted orientation. In particular, homologically essential knots won't be created in the process of attaching $1$-handles.  More explicitly, the twisted orientations on the two arcs where we attach the $1$ handle are as shown in Figure \ref{fig:local_band_addition}. Adding the $1$ handle changes a twisted oriented resolution to another twisted oriented resolution. There are three possible cases for how these two arcs are connected in the twisted orientation resolution:
	\begin{enumerate}
		\item The two arcs $\alpha\beta $ and $\gamma \delta $ belong to the same circle in the twisted orientation resolution, and $\alpha$ is adjacent to $\gamma$ on the circle. Then the induced map by the $1$-handle attachment is $\Delta$, such that \[\Delta(a) = a\otimes a, \,\,\,\Delta (b) = b\otimes b,\] which proves $F_{\Sigma}([s_o]) = [s_{o'}]$ in this case.
		\item The two arcs $\alpha\beta $ and $\gamma \delta $ belong to the same circle in the twisted orientation resolution, but $\alpha $ is adjacent to $\delta $ on the circle. Then the induced map by the $1$-handle attachment is $f$, such that \[f(a) =a, f(b)=b,\] so $F_{\Sigma}([s_o]) = [s_{o'}]$ as well.
		\item The two arcs $\alpha\beta $ and $\gamma \delta $ belong to different circles in the twisted oriented resolution. The labels on these two circles in $s_o$ will be the same, by the rule in Definition \ref{def:assignment of a and b}. Then the induced map induced map by the $1$-handle attachment is $m$, such that \[m(a\otimes a )=a ,\,\,\,\,m(b\otimes b) =b, \] so $F_{\Sigma}([s_o]) = [s_{o'}]$.
	\end{enumerate}
		\begin{figure}[h]
		\[{
			\fontsize{7pt}{9pt}\selectfont
			\def\svgscale{0.9}
			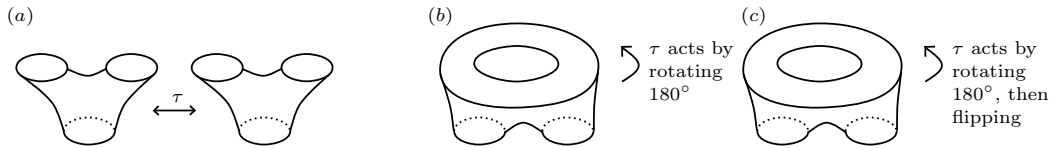
		}\]
		\caption{Action of $\tau$ on a pair of equivariant $1$-handles}
		\label{fig:action_on_1_handles}
	\end{figure}
In terms of the set of twisted orientations on $\Sigma_i$ extending the given one on the boundary, we note the following. If $\Sigma_i$ represents a $1$-handle attachment that splits one circle into two circles, then it is topologically a pair-of-pants. A priori, there are two different possibilities for the covering $\widetilde{\Sigma_i}$ as shown in Figure \ref{fig:action_on_1_handles} $(a)$ and $(b)$. It is either a disjoint union of two pairs-of-pants where the action $\tau$ switches the two components, as shown in (a), or an annulus with two disks removed, as shown in (b), where the action of $\tau$ is a rotation by $180^{\circ}$. However, $\tau$ acts in an orientation-preserving way in case (b), so this possibility is ruled out by the assumption that $\Sigma$ is twisted orientable. The situation for a $1$-handle attachment that merges two circles into one circle is exactly the same, by turning the cobordism upside down. For a $1$-handle attachment twisting one circle into another, the covering $\Sigma_i$ is again an annulus with two disks removed, as shown in (c) of Figure \ref{fig:action_on_1_handles}, while the action $\tau$ is different: It is the usual covering map from an annulus to a M\"obius band.

After this clarification, the rest proof follows exactly in the same way as in \cite{MR2729272}. For $\Sigma_i$ representing a $1$-handle attachment splitting a circle into two or twisting a circle, there is a unique twisted orientation on $\Sigma_i$ which extends the twisted orientation on the input boundary. If $\Sigma_i$ represents a $1$-handle attachment merging two circles into one, then depending on which components these two circles belong to, either all twisted orientations on the input have a unique extension to $\Sigma_i$, or half of them have a unique extension, in which case \[F_{\Sigma_i}(s_{o})=0,\] for twisted orientations $o$ on the input which doesn't extend to $\Sigma_i$.
\end{proof}

\begin{remark}
	 We won't discuss the functoriality issue of cobordism maps in this paper, $i.e.$ whether the map $F_{\Sigma}$ is well-defined up to filtered chain homotopy. The existence of such map $F_{\Sigma}$ with the stated property is enough to prove some genus bound.
\end{remark}

If we restrict to connected cobordisms between null homologous knots, we obtain the following immediate corollary, which is what we need for the genus bound.
\begin{corollary}
	 If $L$, $L'$ are both null homologous knots and $\Sigma$ is a connected, twisted orientable cobordism between them in $\rp{3}\times I$, then \[F_{\Sigma}(\left[s_o\right]) = \left[s_{o'}\right],\] where $o'$ is the restriction of the unique twisted orientation on $\Sigma$ which extends $o$, and \[F_{\Sigma}: \mathit{HBN}(L) \to \mathit{HBN}(L')\] is an isomorphism of filtration degree $\chi(\Sigma)$.
	 
	 \label{coro:chi bound}
\end{corollary}
	 
	\section{The $s$-invariant and genus bound}
\label{sec: invariant}
In the previous two sections, we have proven all the formal properties of the Bar-Natan homology required to define the $s$-invariant and bound slice genus. We summarize the results in this section, following exactly the same procedure as in the case of Bar-Natan or Lee homology in $S^3$. The only difference is that the class we obtain a genus bound for is the class of twisted orientable slice surfaces of null homologous knot in $\rp{3}\times I$ instead of orientable slice surfaces. The treatment here follows the summary of the $s$-invariant for Bar-Natan homology in $S^3$ in Section 2, \cite{MR3189434} closely.


Suppose $L$ is a twisted oriented null homologous link in $\rp{3}$. Denote its Bar-Natan chain complex by $C_{*,*}(L) = \mathit{CBN}_{*,*}(L)$. Consider the finite length filtration filtration by the quantum grading on $C(L)$: \[0 \subset ...\subset \mathcal{F}_{q+1}C(L)\subset \mathcal{F}_qC(L)\subset\mathcal{F}_{q-1}C(L)\subset ... \subset C(L),\]
where $\mathcal{F}_qC(L) = \bigoplus_{j\geq q}C_{*,j}(L)$. Note that due to the possible existence of $1$-$1$ bifurcation in the edge map, which twists one circle into another circle,  the parity of the quantum grading is no longer the same on the chain complex $C_{*,*}(L)$. Therefore, we increase the filtration degree $q$ by $1$ at each step, instead of by $2$ as in the case of Bar-Natan chain complex for links in $S^3.$ The differential map of the chain complex $C$ respects the filtration in the quantum grading, as in the usual case for $m, \Delta$, and we choose the map $f:V \to V\left[1\right]$ to be filtered in its definition, so each $\mathcal{F}_qC(L)$ is a subcomplex.

\begin{definition}
Let $K$ be a twisted oriented null homologous knot in $\rp{3}$. Define the numbers \begin{align*}
	&s_{min}(K) = \text{max}\left\{q\in\mathbb{Z} \mid i_{*}:H(\mathcal{F}_qC(K))\to H(C(K))\cong \mathbb{F}^2 \text{ is surjective}\right\},\\
	&s_{max}(K) = \text{max}\left\{q\in\mathbb{Z} \mid i_{*}:H(\mathcal{F}_qC(K))\to H(C(K))\cong \mathbb{F}^2 \text{ is nonzero}\right\}.
\end{align*}
\label{def:smin and smax}
\end{definition}
\begin{lemma}
	We have $s_{max}(K) = s_{min}(K) +2$.
\end{lemma}

\begin{proof}
The same argument for Proposition 2.6 in \cite{MR3189434} works here with a slight change, so we just give a sketch. 

Consider the involution $I:C(K)\to C(K)$, which is induced by the involution on the Frobenius algebra $V$: \[I(a) = b,\,\,\,\, I(b)=a.\] In terms of the basis $\left\{1,x\right\}$, it is \[I(1) = 1, \,\,\,\, I(x) = 1+x.\] So $I$ induces the identity map on the associated graded complex.  

Choose a cycle $y\in \mathcal{F}_{s_{min}(K)}C(K)$, such that  $\left\{y, I(y)\right\}$ gives a basis of $H(C(K))$. Since the lowest grading parts of $a$ and $I(a)$ are the same, and the grading non-preserving part of the map $I$ on $C(K)$ raises the grading by at least $2$, we have $y+I(y) \in \mathcal{F}_{s_{min}(K)+2}C(K)$. Hence,
\[s_{max}(K) \geq s_{min}(K) +2. \]

The proof of the inequality in the other direction follows exactly the same argument as for the $s$-invariant of the usual Bar-Natan homology, as taking connected sum with a local unknot is a local operation, and the Bar-Natan chain complex behaves in the same way as in the case of $S^3$ with respect to this local operation.
\end{proof}

Hence, we can define the Bar-Natan $s$-invariant for null homologous knot in $\rp{3}$ as usual.
\begin{definition}
	Let $K$ be a null homologous knot in $\rp{3}$. The \textbf{Bar Natan $s$-invariant } of $K$, is defined as \[s_{\rp{3}}^{BN}(K) = \dfrac{s_{min}(K) +s_{max}(K)}{2}.\]\label{def:s}
\end{definition}

It is well-defined as we have checked that the maps induced by Reidemeister moves are filtered maps of filtration degree $0$ in Proposition \ref{prop:R moves}. Also, it doesn't depend on the twisted orientation on $K$, as the effect of reversing twisted orientation is the same as switching $a$ and $b$.

This $s$-invariant $s^{BN}_{\rp{3}}$ satisfies similar properties as the usual $s$-invariants. We summarize some of them here. Compare also with corresponding statements about the $s$-invariant defined using the Lee deformation in \cite{manolescu2023rasmussen}.

\begin{proposition}
	If $K$ is a local knot in $\rp{3}$, $i.e.,$ it is contained in some ball $B^3$ in $\rp{3}$, then \[s_{\rp{3}}^{BN}(K) = s^{BN}(K),\] where $s^{BN}(K)$ denote the $s$-invariants from the Bar-Natan homology for knot in $S^3$ over the field $\mathbb{F} =\mathbb{F}_2$.
\end{proposition}

\begin{proof}
	A local knot is, in particular, null homologous, so we can define $s^{BN}_{\rp{3}}(K)$ for it. For a local knot, the notion of twisted orientation agrees with the usual notion of orientation, as we can draw a knot diagram of it that doesn't intersect the essential unknot $U_1$ at all, so no reversing of the arrow is needed. Therefore, the notion of $n_+$ and $n_-$ with respect to the twisted orientation agrees with the usual notion of $n_+$ and $n_-$. Also, there will be no $1$-$1$ bifurcation appearing in the Bar-Natan chain complex for the knot diagram away from $U_1$. Thus, the notion $s_{\rp{3}}^{BN}$ matches exactly with the usual notion of $s^{BN}(K)$ viewing $K$ as a knot in $S^3$.
\end{proof}

\begin{proposition}
	Let $m(K)$ be the mirror of a null homologous knot in $\rp{3}$, $i.e.$, it is obtained by switching positive crossings with negative crossings in a knot diagram of $K$. Then 
	\label{prop:mirroring}
	\[s^{BN}_{\rp{3}}(m(K)) = -s^{BN}_{\rp{3}}(K).\]
\end{proposition}

\begin{proof}
	The usual argument as in Proposition 3.9 of \cite{MR2729272} using the dual chain complex works here as well, with the dual map of $f = id_{V}$ being the identity $f^* = id_{V^*}$ on $V^{*}$. 
\end{proof}

\begin{proposition}
	If $K$ is a null homologous knot in $\rp{3}$ and $K_l$ is a local knot in $\rp{3}$, then we can form the connected sum \[K\#K_l\subset \rp{3}\#S^3 \cong \rp{3},\] and we have \[s^{BN}_{\rp{3}}(K\#K_l) = s^{BN}_{\rp{3}}(K) + s^{BN}(K_l),\]
	where again $s^{BN}(K_l)$ is the $s$-invariant of $K_l$ using the Bar-Natan homology in $S^3$ over the field $\mathbb{F} = \mathbb{F}_2$.
	\label{prop:local connected sum}
\end{proposition}

\begin{proof}
	The same proof for Proposition 3.11 in \cite{MR2729272} works. Note that a twisted orientation on $K$ and an orientation on $K_l$ gives a twisted orientation on $K\#K_l$, if they are compatible on the two arcs where the connected sum is performed.
\end{proof}
\begin{proposition}
	If $K$ is positive with respect to the twisted orientation, $i.e.$ $n_-=0$ counted using a twisted orientation on $K$, then we have the usual formula for its $s^{BN}_{\rp{3}}$-invariant: \[s^{BN}_{\rp{3}}(K) = -k+n+1,\] where $n$ is the number of crossings of $K$, and $k$ is the number of circle in the twisted oriented resolution of $K$.  
	\label{prop:positive}
\end{proposition}
\begin{proof}
	Exactly the same proof as in Section 5.2 of \cite{MR2729272} works as well, changing everywhere the word 'orientation' by 'twisted orientation'.
\end{proof}

Now we discuss the genus bound one can obtain using the $s$-invariant $s^{BN}_{\rp{3}}$. 

\begin{definition}
		\label{def:twisted orientable slice surface}
	Let $K$ be a null homologous knot in $\rp{3}$. A surface $\Sigma$ in $\rp{3} \times I$ is a \textbf{twisted orientable slice surface} of $K$ if it is connected, twisted orientable, and \[\partial \Sigma =\Sigma \cap (\rp{3}\times \left\{0\right\}) =K.\]
\end{definition}

The most straightforward result one can write down is in terms of the Euler characteristics.

\begin{proposition}
	If $\Sigma $ is a twisted orientable slice surface of a null homologous knot $K$ in $\rp{3}$, then \[-\chi(\Sigma)\geq \abs{s^{BN}_{\rp{3}}(K)}.\]
\end{proposition}

\begin{proof}
	It is a straightforward corollary of Corollary \ref{coro:chi bound} applied to cobordisms from $K$ to the trivial unknot $U$, and the fact $s^{BN}_{\rp{3}}(U) =0$. Check Corollary 2.7 in \cite{MR3189434} for a detailed explanation.
\end{proof}

The genus bound will depend on whether the twisted orientable slice surface is orientable or not in the usual sense, as the formulas for the Euler characteristic from the genus differ in these two case.

\begin{corollary}
	If $\Sigma $ is a twisted orientable slice surface of a null homologous knot $K$ in $\rp{3}$, then 
	
	\begin{equation*}
		g(\Sigma)\geq 
		\begin{cases}
			&\abs{s^{BN}_{\rp{3}}(K)},  \,\,\,\,\,\,\,\,\,\text{ if $\Sigma$ is unorientable;}\\
			&\dfrac{1}{2}\abs{s^{BN}_{\rp{3}}(K)},\,\,\,\,  \text{ if $\Sigma$ is orientable.}\\
		\end{cases}       
	\end{equation*}
\end{corollary}

As the class of twisted orientable slice surface is different from the usual notion of slice genus in $\rp{3}\times I$, it is natural to expect that this $s^{BN}_{\rp{3}}(K)$ gives different information than the $s$-invariant defined for knots in $\rp{3}$ using the Lee deformation as in \cite{manolescu2023rasmussen}. We illustrate the difference in the following example.
		\begin{figure}[h]
	\[{
		\fontsize{7pt}{9pt}\selectfont
		\def\svgscale{0.6}
\begingroup%
  \makeatletter%
  \providecommand\color[2][]{%
    \errmessage{(Inkscape) Color is used for the text in Inkscape, but the package 'color.sty' is not loaded}%
    \renewcommand\color[2][]{}%
  }%
  \providecommand\transparent[1]{%
    \errmessage{(Inkscape) Transparency is used (non-zero) for the text in Inkscape, but the package 'transparent.sty' is not loaded}%
    \renewcommand\transparent[1]{}%
  }%
  \providecommand\rotatebox[2]{#2}%
  \newcommand*\fsize{\dimexpr\f@size pt\relax}%
  \newcommand*\lineheight[1]{\fontsize{\fsize}{#1\fsize}\selectfont}%
  \ifx\svgwidth\undefined%
    \setlength{\unitlength}{687.44438144bp}%
    \ifx\svgscale\undefined%
      \relax%
    \else%
      \setlength{\unitlength}{\unitlength * \real{\svgscale}}%
    \fi%
  \else%
    \setlength{\unitlength}{\svgwidth}%
  \fi%
  \global\let\svgwidth\undefined%
  \global\let\svgscale\undefined%
  \makeatother%
  \begin{picture}(1,0.61982019)%
    \lineheight{1}%
    \setlength\tabcolsep{0pt}%
    \put(0,0){\includegraphics[width=\unitlength,page=1]{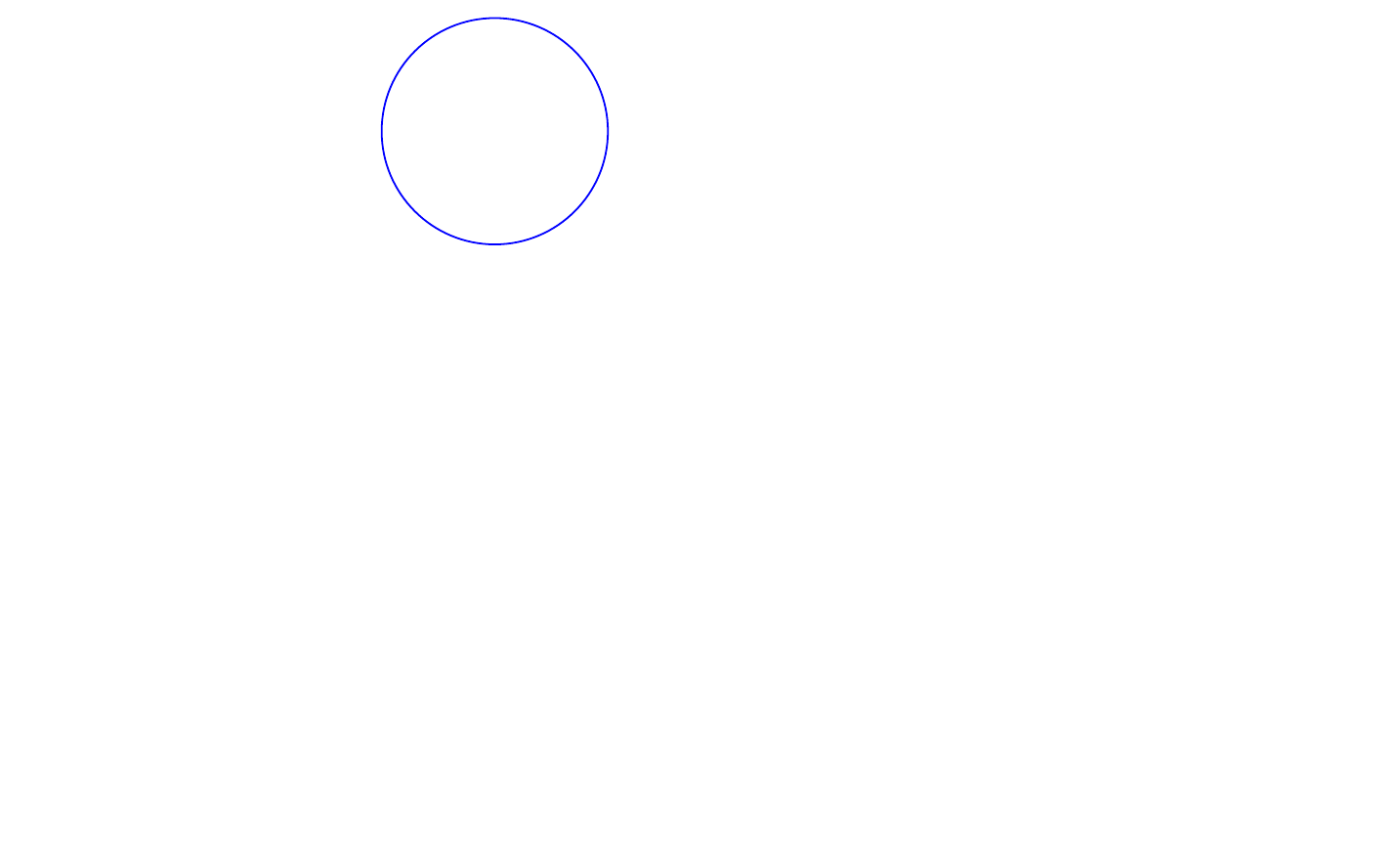}}%
    \put(-0.00055828,0.61113175){\color[rgb]{0,0,0}\makebox(0,0)[lt]{\lineheight{1.25}\smash{\begin{tabular}[t]{l}$(a)$\end{tabular}}}}%
    \put(-0.00055828,0.38690666){\color[rgb]{0,0,0}\makebox(0,0)[lt]{\lineheight{1.25}\smash{\begin{tabular}[t]{l}$(b)$\end{tabular}}}}%
    \put(-0.00055828,0.14451153){\color[rgb]{0,0,0}\makebox(0,0)[lt]{\lineheight{1.25}\smash{\begin{tabular}[t]{l}$(c)$\end{tabular}}}}%
    \put(0,0){\includegraphics[width=\unitlength,page=2]{example_calculation.pdf}}%
    \put(0.475,0.59){\color[rgb]{0,0,0}\makebox(0,0)[lt]{\lineheight{1.25}\smash{\begin{tabular}[t]{l}twisted \\oriented\\resolution\end{tabular}}}}%
  \end{picture}%
\endgroup%

	}\]
	\caption{An example of difference in twisted orientable slice surface and usual orientable slice surface}
	\label{fig:example_calculation}
\end{figure}
\begin{example}
	\label{ex:a calculation}
	Consider the following null homologous knot $K$ as drawn in part $(a)$ of  Figure \ref{fig:example_calculation}. It is a a positive knot with respect to the twisted orientation, with $4$ crossings, and $2$ circles in the twisted oriented orientation, so by Proposition \ref{prop:positive}, we have \[s^{BN}_{\rp{3}} = -2+4+1=3.\] A twisted orientable slice surface of $K$ with $\chi = -3$ is drawn in part $(b)$ of Figure \ref{fig:example_calculation} by adding $3$ bands compatible with the twisted orientation. 
	
	Therefore, a twisted orientable slice surface $\Sigma$ of $K$ which is actually orientable should have genus \[g(\Sigma) \geq \dfrac{3}{2,}\] so at least of genus $2$. But there exists a orientable slice surface of $K$ in the usual sense, which is of genus $1$, which is drawn in part $(c)$ of Figure \ref{fig:example_calculation}. Note that the arrows in part $(b)$ represent a twisted orientation, while the arrows in part $(c)$ represent a usual orientation.
	
	By inserting more and more full twists in this example, we have a family of null homologous knots $K_n$ in $\rp{3}$, whose two $s$-invariants $s^{BN}_{\rp{3}}(K)$ and $s^{Lee}_{\rp{3}}(K)$ defined using the Bar-Natan and Lee deformation, respectively, could have arbitrary large differences in their absolute values. Here $s^{Lee}_{\rp{3}}(K)$ refers to the $s$-invariant defined in \cite{manolescu2023rasmussen} using the Lee deformation. See Figure \ref{fig:infinite_difference} for a diagram of $K_n$. It is a positive knot with respect to the twisted orientation and a negative knot with respect to the usual orientation. It has $2n+2$ crossings, and the twisted oriented resolution has $2$ circles, while the oriented resolution has $2n+1$ circles Therefore,
	\[\abs{s^{BN}_{\rp{3}}(K_n)} = \abs{-2+(2n+2)+1}=2n+1\,\,\text{ and }\,\, \abs{s^{Lee}_{\rp{3}}(K_n)} = \abs{-(-(2n+1)+(2n+2)+1)} =2.\]
\end{example}
		\begin{figure}[h]
	\[{
		\fontsize{7pt}{9pt}\selectfont
		\def\svgscale{0.6}
\begingroup%
  \makeatletter%
  \providecommand\color[2][]{%
    \errmessage{(Inkscape) Color is used for the text in Inkscape, but the package 'color.sty' is not loaded}%
    \renewcommand\color[2][]{}%
  }%
  \providecommand\transparent[1]{%
    \errmessage{(Inkscape) Transparency is used (non-zero) for the text in Inkscape, but the package 'transparent.sty' is not loaded}%
    \renewcommand\transparent[1]{}%
  }%
  \providecommand\rotatebox[2]{#2}%
  \newcommand*\fsize{\dimexpr\f@size pt\relax}%
  \newcommand*\lineheight[1]{\fontsize{\fsize}{#1\fsize}\selectfont}%
  \ifx\svgwidth\undefined%
    \setlength{\unitlength}{329.5131564bp}%
    \ifx\svgscale\undefined%
      \relax%
    \else%
      \setlength{\unitlength}{\unitlength * \real{\svgscale}}%
    \fi%
  \else%
    \setlength{\unitlength}{\svgwidth}%
  \fi%
  \global\let\svgwidth\undefined%
  \global\let\svgscale\undefined%
  \makeatother%
  \begin{picture}(1,0.82312809)%
    \lineheight{1}%
    \setlength\tabcolsep{0pt}%
    \put(0,0){\includegraphics[width=\unitlength,page=1]{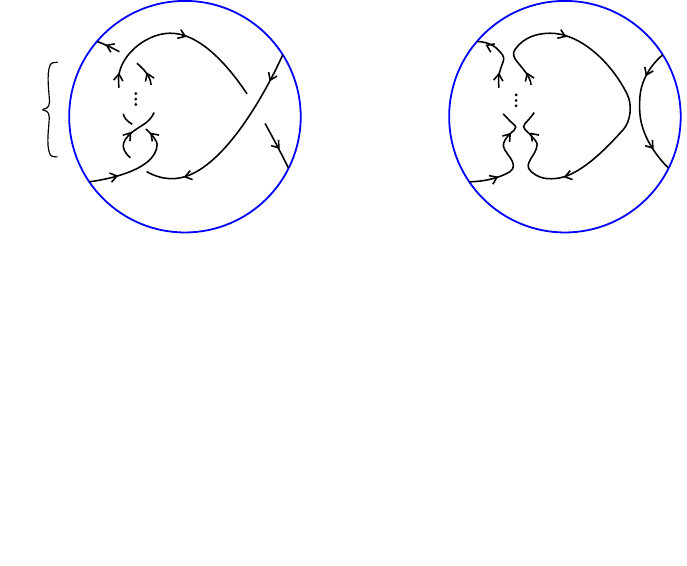}}%
    \put(-0.1,0.6550721){\color[rgb]{0,0,0}\makebox(0,0)[lt]{\lineheight{1.25}\smash{\begin{tabular}[t]{l}$2n+1$\\crossings\end{tabular}}}}%
    \put(0,0){\includegraphics[width=\unitlength,page=2]{infinite_difference.pdf}}%
    \put(-0.1,0.21010217){\color[rgb]{0,0,0}\makebox(0,0)[lt]{\lineheight{1.25}\smash{\begin{tabular}[t]{l}$2n+1$\\crossings\end{tabular}}}}%
    \put(0.02,0.42){\color[rgb]{0,0,0}\makebox(0,0)[lt]{\lineheight{1.25}\smash{\begin{tabular}[t]{l}$K_n$ with a twisted orientation\end{tabular}}}}%
    \put(0.02,-0.01){\color[rgb]{0,0,0}\makebox(0,0)[lt]{\lineheight{1.25}\smash{\begin{tabular}[t]{l}$K_n$ with a usual orientation\end{tabular}}}}%
    \put(0.65,0.42){\color[rgb]{0,0,0}\makebox(0,0)[lt]{\lineheight{1.25}\smash{\begin{tabular}[t]{l}a twsited oriented resolution of $K_n$\end{tabular}}}}%
    \put(0.67,-0.01){\color[rgb]{0,0,0}\makebox(0,0)[lt]{\lineheight{1.25}\smash{\begin{tabular}[t]{l}an oriented resolution of $K_n$\end{tabular}}}}%
  \end{picture}%
\endgroup%

	}\]
	\caption{An example of knots with different $\abs{s^{BN}_{\rp{3}}}$ and $\abs{s^{Lee}_{\rp{3}}}$}
	\label{fig:infinite_difference}
\end{figure}

\section{Further directions}
\label{sec:further directions}
In this section we discuss some further directions to explore, listed in ascending order of scope.

\begin{itemize}

	\item 

In \cite{manolescu2023rasmussen}, Manolescu and Willis defined the Lee homology and $s$-invariant for both null homologous and homologically essential links (class-$0$ and class-$1$ links in their notion), while we only define the Bar-Natan homology and $s$- invariant for null homologous links in this paper. It is natural to ask what the counterpart for homologically essential links should be using Bar-Natan homology. Note that the algebraic structures of Khovanov homology (and Lee homology) for null homologous and homologically essential links are quite different: In any resolution of the link diagram of a homologically essential link, there will be a homologically essential unknot, and there will be no $1$-$1$ bifurcation at all in the edge maps. So instead of the extra map $f:V\to V$, what we want in the homologically essential case is a bimodule over the Frobenius algebra $V$, which will be assigned to the homologically essential unknot in each resolution.

\item
Naturally, one would like to ask how this Bar-Natan chain complex $\mathit{CBN}(L)$ is related to some equivariant version of the Bar-Natan chain complex $\mathit{CBN}(\widetilde{L})$ of the double cover $\widetilde{L}$ of $L$ in $S^3$, and also the similar question for the Khovanov chain complex. When $\widetilde{L}$ is periodic or strongly invertible, there is a lot of work relating the Khovanov chain complex/stable homotopy type of the quotient link to the corresponding chain complex/stable homotopy type of the link itself with the action. See, for example, \cite{stoffregen2022localization}, \cite{MR4297186} for periodic links, and \cite{MR4286365}, \cite{lipshitz2022khovanov} for strongly invertible knots. Since the Bar-Natan homology is much easier to describe than the Khovanov homology, we do obtain an inequality at the level of Bar-Natan homology: \[dim(\mathit{HBN}(L)) \leq dim(\mathit{HBN}(\widetilde{L})),\] where $\mathit{HBN}(L)$ is the Bar-Natan homology of a null homologous link $L$ in $\rp{3}$ defined in this paper, and $\mathit{HBN}(\widetilde{L})$ is the usual Bar-Natan homology for links in $S^3$. The reason is simple: As proved in Proposition \ref{prop:canonical generator},  $\mathit{HBN}(L)$ has a basis identified with the set of twisted orientations on $L$, which is by definition a subset of orientations on $\widetilde{L}$, and $\mathit{HBN}(\widetilde{L})$ has a basis identified with the set of all orientations on $\widetilde{L}$. This suggests a possibility of a spectral sequence relating $\mathit{HBN}(L)$ and $\mathit{HBN}(\widetilde{L})$, and perhaps for the Khovanov homology as well. One possible starting point is to look at a link diagram of $\widetilde{L}$ as the closure of $T\circ F(T)$, as in Figure \ref{fig:twisted-orientation_in_double_cover}, on which one can define formally an action of the involution $\tau$ on the chain complex level for $\mathit{CBN}(\widetilde{L})$ and $CKh(\widetilde{L})$. 
\item
For knots in $S^3$, the Lee and Bar-Natan deformation of the usual Khovanov homology are closely related to each other. For example, if $\mathbb{F}$ is a field of characteristic other than $2$, then the $s$-invariant defined using the Bar-Natan deformation agrees with that defined using the Lee deformation. See Proposition 3.1 in \cite{MR2320159}.  However, in $\rp{3}$, the behavior is quite different. For example, the class of slice surface whose genus bounded by the $s$-invariant is not the same. The reason for this difference lies in the assignment of the map to the $1$-$1$ bifurcation: it is the identity map in the Bar-Natan deformation, while it is $0$ in the Lee deformation. It might be interesting to look at other extensions of the $s$-invariants to $3$-manifolds using the Lee deformation, \textit{e.g.} \cite{MR3882203} for $S^1\times D^2$, \cite{manolescu2022generalization} for connected sums of $S^1\times S^2$, and see what happens if one use the Bar-Natan deformation instead of the Lee deformation in these cases. 

\end{itemize}
\bibliographystyle{amsalpha}
\bibliography{biblio}
\end{document}